\documentclass{article}
\usepackage[a4paper]{geometry} 

\usepackage{easymat}
\usepackage{mathrsfs}
\usepackage{verbatim}
\usepackage{amsmath}
\usepackage{amssymb}
\usepackage{amsfonts}
\usepackage{mathrsfs}
\usepackage{graphicx}
\usepackage{tikz}
\usetikzlibrary{shapes.misc, positioning}
\usetikzlibrary{arrows}
\usepackage{placeins}
\usepackage{float}
\usepackage{nicefrac}
\usepackage{array}
\usepackage{stmaryrd}
\usepackage{braket}
\usepackage{enumitem}
\usepackage{caption}
\usepackage{subcaption}
\usepackage{multirow}
\usepackage{array}
\usepackage{blkarray}
\usepackage{ctable}
\usepackage{arydshln}

\usepackage{import}

\usepackage[round]{natbib}
\usepackage{bibentry}
\nobibliography*

\usepackage{adjustbox}

\usepackage{xstring}

\usepackage{dsfont}
\usepackage{algorithm}
\usepackage{algorithmic}
\usepackage{siunitx}
\usepackage{indentfirst}
\usepackage{prodint} 
\usepackage{scalerel} 

\usetikzlibrary{positioning, decorations.pathreplacing} 

\usepackage{authblk} 
\usepackage{hyphenat} 

\DeclareMathOperator*{\argmin}{arg\,min}

\NewDocumentCommand{\Esp}{}{\mathbb{E}}
\NewDocumentCommand{\Prob}{}{\mathbb{P}}
\NewDocumentCommand{\Var}{}{\mathrm{Var}}

\NewDocumentCommand{\D}{}{\mathcal{D}}

\NewDocumentCommand{\M}{}{\mathcal{M}}

\NewDocumentCommand{\Rbb}{}{\mathbb{R}}
\NewDocumentCommand{\tauH}{}{\tau_{H}}

\NewDocumentCommand{\Fn}{o}{\IfNoValueTF{#1}{\hat{F}_n}{\hat{F}_{n_#1}}}
\NewDocumentCommand{\Sn}{o}{\IfNoValueTF{#1}{\hat{S}_n}{\hat{S}_{n_#1}}}
\NewDocumentCommand{\Gn}{o}{\IfNoValueTF{#1}{\hat{G}_n}{\hat{G}_{n_#1}}}
\NewDocumentCommand{\Ln}{o}{\IfNoValueTF{#1}{\hat{\Lambda}_n}{\hat{\Lambda}_{n_#1}}}

\NewDocumentCommand{\Dn}{o}{\IfNoValueTF{#1}{D_n}{D_{n_#1}}}

\NewDocumentCommand{\indicator}{m}{\mathds{1}\{#1\}}

\NewDocumentCommand{\muThat}{o}{\IfNoValueTF{#1}{\hat{\mu}_{\tau,n}}{\hat{\mu}_{\tau,n_#1}}}

\NewDocumentCommand{\intT}{}{\int_0^{\tau}}

\NewDocumentCommand{\khat}{}{\widehat{k}}
\NewDocumentCommand{\ktilde}{}{\widetilde{k}}
\NewDocumentCommand{\Hhat}{m}{\widehat{H}_{#1}}
\NewDocumentCommand{\Htilde}{m}{\widetilde{H}_{#1}}

\NewDocumentCommand{\PnBn}{}{P_{B_n}}
\NewDocumentCommand{\psihat}{}{\widehat{\psi}}
\NewDocumentCommand{\thetahat}{}{\widehat{\theta}_{n}}
\NewDocumentCommand{\thetatilde}{}{\widetilde{\theta}_{n}}
\NewDocumentCommand{\EspBn}{}{\Esp_{B_n}}

\NewDocumentCommand{\Pni}{}{P^{\otimes n_1}}
\NewDocumentCommand{\po}{}{^{\text{\normalfont po}}}
\NewDocumentCommand{\thetahatpo}{}{\widehat{\theta}_{n}\po}
\NewDocumentCommand{\thetatildepo}{}{\widetilde{\theta}_{n}\po}
\NewDocumentCommand{\xin}{}{\xi_{n_1}}

\NewDocumentCommand{\bi}{}{\big}
\NewDocumentCommand{\bii}{}{\Big}
\NewDocumentCommand{\biii}{}{\bigg}
\NewDocumentCommand{\biiii}{}{\Bigg}
\NewDocumentCommand{\midi}{}{\bigm|}
\NewDocumentCommand{\midii}{}{\Bigm|}

\usepackage{amsthm}

\newtheorem{theorem}{Theorem}

\newtheorem{corollary}{Corollary}[theorem]

\newtheorem*{théorème*}{Théorème}
\newtheorem*{définition*}{Définition}
\newtheorem{corollaryF*}{Corollaire}
\newtheorem{proposition*}{Proposition}
\newtheorem*{propriété*}{Propriété}

\newtheorem*{lemma*}{Lemma}
\newtheorem{lemma}{Lemma}

\newtheorem{assumption}{Assumption}
\newtheorem*{hypothèse*}{Hypothèse}

\theoremstyle{remark}
\newtheorem{ex}{Example}

\title{Pseudo-Observations and Super Learner for the \\ Estimation of the Restricted Mean Survival Time}
\author[]{Ariane Cwiling}
\author[]{Vittorio Perduca}
\author[]{Olivier Bouaziz}
\affil[]{Université Paris Cité, CNRS, MAP5, F-75006 Paris, France}
\date{}

\begin{document}

\maketitle

\begin{abstract}
    In the context of right-censored data, we study the problem of predicting the restricted time to event based on a set of covariates. Under a quadratic loss, this problem is equivalent to estimating the conditional Restricted Mean Survival Time (RMST). To that aim, we propose a flexible and easy-to-use ensemble algorithm that combines pseudo-observations and super learner. The classical theoretical results of the super learner are extended to right-censored data, using a new definition of pseudo-observations, the so-called split pseudo-observations. Simulation studies indicate that the split pseudo-observations and the standard pseudo-observations are similar even for small sample sizes. The method is applied to maintenance and colon cancer datasets, showing the interest of the method in practice, as compared to other prediction methods. We complement the predictions obtained from our method with our RMST-adapted risk measure, prediction intervals and variable importance measures developed in a previous work.
    
    \textbf{Keywords:} Right-censoring, RMST, prediction, stacking, pseudo-observations, super learner.
\end{abstract}

\section{Introduction}
Predicting the time to an event of interest based on a set of covariates is a relevant goal in applications. For example, in medical applications, it could be interesting to predict the time to onset of cancer, relapse or death of a patient, while in industrial applications, we might be interested in predicting the time to failure of a mechanical part. Due to tail issues caused by right-censoring, it is common to focus on the restricted time to event instead of the time to event itself~\citep[see][]{eaton_designing_2020}. In practice, this prediction problem is equivalent to estimating the Restricted Mean Survival Time (RMST) when a quadratic loss is used. The RMST is a clinically meaningful quantity that has gained attention over the years  for its simplicity and interpretability. The RMST can be easily retrieved by integrating an estimator of the survival function, yet new approaches have been developed to directly model it~\citep[see][]{andersen_regression_2004, tian_predicting_2014, wang_modeling_2018}. Remarkably, these methods avoid strong modeling assumptions such as the proportional hazard assumption from the Cox model. 

In particular, pseudo-observations, introduced by~\cite{andersen_regression_2004}, have enabled the application of a large range of RMST estimation models. Pseudo-observations are a transformation of the incomplete observed times that have the strong following property~\citep[see][]{jacobsen_note_2016}: Their conditional expectation is equal, up to a remainder term, to the conditional RMST. As a result, they can be fed into any prediction model adapted to uncensored data, from generalized linear models~\citep{andersen_regression_2004} to neural networks~\citep{zhao_deep_2021}. In practical applications, the choice of the best learner among all of available machine learning algorithms for uncensored data is a difficult question, as the prediction performance of the different algorithms may depend on the application. As a result, an attractive alternative is to use ensemble methods, such as the super learner algorithm~\citep[see][]{van_der_laan_super_2007}. When dealing with uncensored data, it has been proven that, under mild assumptions, the learner with the lowest cross-validated risk in a user-defined library of algorithms, the so-called discrete super learner, performs as well as the best algorithm in the library up to an error term~\citep{van_der_laan_unified_2003, dudoit_asymptotics_2005, van_der_vaart_oracle_2006}. This result can be exploited to construct the best linear combination of all the learners in the library, the so-called continuous super learner. This stacking method has the theoretical guarantee to perform asymptotically at least as well as any of the candidate learners. The continuous super learner is simply termed ``super learner'' in what follows.

The goal of this work is to adapt the super learner algorithm for RMST estimation with right-censored data. The super learner for analyzing censored data has already been the subject of previous research. For instance~\citet{golmakani_super_2020} proposed a survival super learner that takes censored data as input and applies a library of survival algorithms. However, the models used in this super learner must verify the proportional hazard assumption, which is a strong limitation of the approach. Another example is the super learner with Inverse Probability Censoring Weight (IPCW) loss~\citep[see][]{van_der_laan_unified_2003, keles_asymptotically_2004, gonzalez_ginestet_stacked_2021, devaux_individual_2022}, which allows to use any survival model as long as the conditional independence assumption holds and the censoring distribution is consistently estimated. Finally, \citet{sachs_ensemble_2019} and \citet{gonzalez_ginestet_survival_2022} proposed two super learner methods with a pseudo-observations-based-AUC loss. However, to our knowledge, none of the pseudo-observations-based super learner methods are provided with theoretical convergence guarantees.

In this work, we present a novel approach combining the super learner with pseudo-observations for the estimation of the RMST.  In order to derive theoretical results similar to those of~\citet{van_der_laan_super_2007} in the uncensored case, we first present a new type of pseudo-observations, namely the \textit{split pseudo-observations}. Those are introduced, along with the standard pseudo-observations, in Section~\ref{sec::pobs}. In Section~\ref{sec::SL}, we present in details the super learner based on both standard and split pseudo-observations. Theoretical results for the combination of split pseudo-observations with the super learner are also derived. The performance of our method is studied extensively through simulations in Section~\ref{sec::simulations} and on two real datasets in Section~\ref{sec::realdata}.

\section{Pseudo-observations} \label{sec::pobs}

In the context of right-censored data, we denote by $T^*$ the time to the event, $C$ the censoring time, $T=T^* \wedge C$ the observed time and $\delta=\indicator{T^*\leq C}$ the censoring indicator. An observation is then represented by the vector $O = (T,\delta,Z)$ where $Z\in\Rbb^d$ is a covariate vector. We note $S(t\mid Z) = \Prob(T^* > t \mid Z)$ the survival function of $T^*$ conditionally on the covariates $Z$.
Let $\tauH = \inf\{t>0 : \Prob(T > t \mid Z)= 0 \text{ a.s.} \}$. The RMST is defined for a fixed time horizon $\tau < \tauH$, conditionally on the covariates, as
\begin{equation*}
     \Esp[T^* \wedge \tau \mid Z] = \intT S(t \mid Z) dt.
\end{equation*}
Given this definition, the RMST can be estimated for instance by integrating an estimator of the conditional survival function between $0$ and $\tau$, or by regressing the restricted event times on covariates. In the second case, censoring must be taken into account since the times $T^*$ are not observed for all individuals. This can be achieved by using pseudo-observations. 
Consider censored observations $D_n=\{O_i = (T_i,\delta_i,Z_i), i=1,\dots,n\}$, where $n$ is the sample size. Classical pseudo-observations are computed in the following way: For a given 
$\tau < \tau_H$ and all $i=1,\dots,n$,
\begin{equation}\label{eq::pobs}
    \Gamma_i := n \int_0^{\tau} \hat{S}(t) dt - (n-1) \int_0^{\tau} \hat{S}^{-i}(t) dt,
\end{equation}
where $\hat{S}$ is the Kaplan-Meier estimator of the survival function computed on all data and $\hat{S}^{-i}$ is the same estimator computed on all data but the $i$-th observation.
The interest of pseudo-observations for regression purposes lies in the following result by \citet{jacobsen_note_2016}:
\begin{equation*}
    \Gamma_i = \Esp[T^* \wedge \tau] + \int_0^{\tau} \Dot{\phi}(O_i)(t) dt + \xi_n,
\end{equation*}
where  $\xi_n = o_\Prob(1)$, $\phi$ is the first order influence function in the Von Mises expansion of the Kaplan-Meier estimator, and
\begin{equation*}
    \Esp\left[\int_0^{\tau} \Dot{\phi}(O_i)(t) dt \midii Z_i\right] = \Esp[T^* \wedge \tau \mid Z_i] - \Esp[T^* \wedge \tau],
\end{equation*}
so that 
\begin{equation}\label{eq::pobs_jacobsen}
    \Esp[\Gamma_i \mid Z_i] = \Esp[T^* \wedge \tau \mid Z_i] + \Esp[\xi_n \mid Z_i].
\end{equation}
This result is valid under the following independent censoring assumption. 
\begin{assumption}[Independent censoring] \label{ass::indep_cens}
    The censoring time $C$ and the pair of variables $(T^*,Z)$ are independent.
\end{assumption}

Pseudo-observations are, by construction, correlated with each other, which makes it difficult to study their theoretical properties. To deal with this issue, 
we propose a new type of pseudo-observations, called \textit{split pseudo-observations}. The idea is to split the data in two subsets $\Dn[1]$ and $\Dn[2]$ of size $n_1$ and $n_2 = n-n_1$, respectively. The former is used to compute the Kaplan-Meier estimator and the latter for the pseudo-observations. We then define a new type of pseudo-observations as follows:
\begin{equation}\label{eq::split_pobs}
    \Gamma_i(\Dn[1]) = \Gamma_{O_i}(\Dn[1]) := (n_1+1) \int_0^{\tau} \hat{S}_{\Dn[1]}^{+i}(t) dt - n_1 \int_0^{\tau} \hat{S}_{\Dn[1]}(t) dt, 
\end{equation}
where $O_i$ is an observation in $\Dn[2]$, $\hat{S}_{\Dn[1]}$ is the Kaplan-Meier estimator of the survival function computed on the $n_1$ data points in $\Dn[1]$ and $\hat{S}_{\Dn[1]}^{+i}$ is the same estimator computed on the $n_1 + 1$ data points obtained by adding $O_i$ to the sample $\Dn[1]$. The main advantage of this construction is that the new pseudo-observations constructed for all the observations in $\Dn[2]$ are independent conditionally on $\Dn[1]$. A result similar to Equation~\eqref{eq::pobs_jacobsen} can then be easily derived for those split pseudo-observations. Under Assumption~\ref{ass::indep_cens}:
\begin{equation}\label{eq::pobs_split_cond_mean_def}
    \Esp[\Gamma_i(\Dn[1]) \mid Z_i,\Dn[1]] = \Esp[T^*_i \wedge \tau \mid Z_i] + \Esp[\xin \mid Z_i, \Dn[1]],
\end{equation}
where $\xin = o_\Prob(1)$. 
From a theoretical standpoint, we establish in Section~\ref{sec::SL} finite sample and asymptotic results for the super learner coupled with split pseudo-observations. In Section~\ref{sec::simulations}, we observe on simulated data that split and traditional pseudo-observations  are very similar, and the choice between them has minimal impact on the prediction quality. Therefore, while split pseudo-observations are easier to study for theoretical results, split and traditional pseudo-observations can be used interchangeably in applications. 

\section{Super Learner} \label{sec::SL}

In this section we introduce a super learner algorithm for right-censored data based on pseudo-observations. Our aim is to predict the restricted time to event with a quadratic loss - the Mean Squared Error (MSE) - using the super learner. It is well known that, under this loss, the best prediction model is the conditional expectation of the restricted time. Therefore, our super learner algorithm can also be seen as an estimator of the RMST. We start by reintroducing the classical super learner for uncensored observations along with a fundamental theoretical result from~\cite{dudoit_asymptotics_2005} and we then derive a new result for our proposed approach.

We can formalize the problem in the following way. Let $\psi : \M \to \D(\mathcal{Z})$ be a parameter mapping a distribution from a statistical model $\M$ to an element in the space $\D(\mathcal{Z})$ of real-valued functions defined on a $d$-dimensional Euclidean set $\mathcal{Z} \subset \Rbb^{d}$. A realization $\psi = \psi(P)$ of $\psi$ for a given $P \in \M$ belongs to the parameter space $\Psi := \{ \psi(P): P \in \M\} \subset \D(\mathcal{Z})$. Consider a loss function
\[
    L : (\psi,O) \mapsto L(\psi,O) \in \Rbb, \text{ for } \psi \in \Psi, O \sim P.
\]
The quantity of interest is the \textit{risk} of the parameter $\psi$ for a distribution $P$,
\[
    \Theta(\psi,P) = \int L(\psi,o) dP(o).
\]
Given this definition, the risk minimizer is defined as
\begin{equation*}
    \psi^* = \psi^*(P) = \argmin_{\psi \in \Psi} \Theta(\psi,P) = \argmin_{\psi \in \Psi} \int L(\psi,o) dP(o),
\end{equation*}
and characterizes the optimal risk
\begin{equation}\label{eq::optimal_risk}
    \theta^* = \Theta(\psi^*,P) = \min_{\psi \in \Psi} \Theta(\psi,P) = \min_{\psi \in \Psi} \int L(\psi,o) dP(o).
\end{equation}

For instance, if $L(\psi,O) = \bi(T^* \wedge \tau - \psi(Z)\bi)^2$ is the quadratic loss function and if \linebreak $P(t,z) = \Prob(T^*\leq t, Z \leq z)$ is the joint law of $(T^*, Z)$,  then the RMST minimizes this risk, i.e. $\psi^*(Z) = \Esp[T^* \wedge \tau \mid Z]$. 
In practice, the risk minimizer is typically unknown and we aim to approximate it as faithfully as possible. 
Suppose that a dataset of $n$ i.i.d. observations $\{O_i, i=1,\dots,n\}$, $O_i \sim P \in \M$, of empirical law $P_n$, is available to estimate a parameter $\psi(P)$. 
An estimator mapping $\psihat$ can be viewed simply as an algorithm one applies to data, i.e. to empirical distributions $P_n$. 
A realization of this mapping applied to a particular empirical distribution $P_n$ is denoted $\psihat(P_n)$ and belongs to $\D(\mathcal{Z})$.
 
\subsection{Uncensored data} For the sake of simplicity, suppose first that we observe the complete data $\{O_i^* = (T_i^*,Z_i), i=1,\dots,n\}$, and we have access to $K_n$ candidate estimators $\psihat_k$, $k=1,\dots,K_n$, of $\psi^*$. We also call these estimators candidate learners. We then wish to select the best estimator among all the candidate learners, that is the one that minimizes the quadratic loss. 
\citet{van_der_laan_unified_2003}, \citet{dudoit_asymptotics_2005} and \citet{van_der_vaart_oracle_2006} theoretically validated the use of cross-validation to select an optimal learner among many candidates and protect against overfitting.
Cross-validation starts with splitting the data into a training set and a validation set. Candidate estimators are then constructed on the training set and evaluated on the validation set. 
Formally, we divide the data according to an independent random vector $B_n = (B_n(i): i = 1,\dots,n) \in \{0,1\}^n$ into a training set $\{O_i^*: i, B_n(i)=0\}$ of size $n_0 = n - \lfloor n p_n  \rfloor$ and a validation set $\{O_i^*: i, B_n(i)=1\}$ of size $n_1 = \lfloor n p_n  \rfloor$ for $p_n \in (0,1)$. 
Several cross-validation schemes, i.e. distributions for $B_n$, can be used. In the following, we will focus on $V$-fold cross-validation, where data are divided into $V$ subsets, or folds, of approximately same size. One by one, each fold serves as validation set while the remaining folds constitute the training set. The associated distribution of $B_n$ assigns a mass of $1/V$ to each of the resulting $V$ binary vectors.

We suppose that all observations $O_i^*$ are i.i.d. of law $P^* \in \M$, and we denote $\PnBn^0$ and $\PnBn^1$ the empirical distributions of the training and validation sets respectively. The cross-validated risk estimator for the $k$-th candidate learner is then defined as
\begin{align*}
    \thetahat(k) & = \EspBn \Theta\bi(\psihat_k(\PnBn^0),\PnBn^1\bi) = \EspBn \int L\bi(\psihat_k(\PnBn^0),o\bi)d\PnBn^1(o) \\
    & = \EspBn \frac{1}{n_1} \sum_{i: B_n(i) =1} L\bi(\psihat_k(\PnBn^0),O_i^* \bi),
\end{align*}
and the cross-validated selector is denoted
\begin{equation*}
    \khat = \argmin_{k \in \{1,\dots,K_n\}} \thetahat(k).
\end{equation*}
Optimality results are based on the comparison between the cross-validated selector and the selector which, for each given dataset, makes the best choice, knowing the true full data distribution.
This cross-validated oracle selector minimizes the cross-validated conditional risk
\begin{equation}\label{eq::cv_risk}
    \thetatilde(k) = \EspBn \Theta\bi(\psihat_k(\PnBn^0),P^* \bi) = \EspBn \int L\bi(\psihat_k(\PnBn^0),o \bi)dP^*(o),
\end{equation}
and is denoted
\begin{equation}\label{eq::cv_oracle_selector}
    \ktilde = \argmin_{k \in \{1,\dots,K_n\}} \thetatilde(k).
\end{equation}
In~\cite{dudoit_asymptotics_2005}, it is proven that the cross-validation selector performs asymptotically as well as the oracle cross-validation selector in terms of performance measure such as the MSE for regression or the AUC for classification tasks. The result for the MSE in the setting with uncensored data is stated in Theorem~\ref{thm::CV_uncensored} below.
\begin{theorem} \label{thm::CV_uncensored}
    Let $O_1^*, \dots, O_n^*$ be a random sample from a data generating distribution $P^*$, where each $O_i^* = (T_i^*,Z_i)$ consists of two components, a univariate outcome $T_i^* \in \Rbb^+$ and a $d$-dimensional covariate vector $Z_i \in \Rbb^d$. Let $\{\psihat_k : k = 1,\dots,K_n\}$ denote a sequence of $K_n$ candidate estimators for the RMST, $\psi^*(Z) = \Esp[T^* \wedge \tau \mid Z]$, which is the risk minimizer for the quadratic loss function $L(\psi,O^*) = \bi(T^* \wedge \tau - \psi(Z)\bi)^2$. Suppose there exists $M$ such that $\tau \leq M < \infty$ and
    \begin{equation*}
         \sup_{Z \in \mathcal{Z},\psi \in \Psi} \lvert \psi(Z) \rvert \leq M \text{ almost surely.}
    \end{equation*}
    \begin{description}
    \item \textbf{Finite sample result.} Let $M_1 = 8 M^2$, $M_2 = 16 M^2$ and $c(M, \gamma) = 2(1+\gamma)^2 (M_1/3 + M_2/\gamma)$. For all $\gamma > 0$, 
    \begin{equation*}
        0 \leq \Esp [\thetatilde(\khat) - \theta^*] \leq (1+2 \gamma) \Esp[ \thetatilde(\ktilde) - \theta^*] + 2 c(M,\gamma) \frac{1 + \log(K_n)}{n p_n}\cdot
    \end{equation*}
    \item \textbf{Asymptotic results.} If 
    \[
        \frac{\log(K_n)}{n p_n \Esp[ \thetatilde(\ktilde) - \theta^*]} \to 0 \text{ as } n \to \infty,
    \]
    then
    \begin{equation*}
        \frac{\Esp[ \thetatilde(\khat) - \theta^*]}{\Esp[ \thetatilde(\ktilde) - \theta^*]} \to 1 \text{ as } n \to \infty.
    \end{equation*}
    Similarly, if 
    \[
        \frac{\log(K_n)}{n p_n (\thetatilde(\ktilde) - \theta^*)} \to 0 \text{ in probability as } n \to \infty,
    \]
    then
    \begin{equation*}
        \frac{\thetatilde(\khat) - \theta^*}{\thetatilde(\ktilde) - \theta^*} \to 1 \text{ in probability as } n \to \infty.
    \end{equation*}
    \end{description}
\end{theorem}

Theoretical guarantees for the discrete super learner algorithm, which selects the learner minimizing the cross-validated risk among all candidates, are outlined in Theorem~\ref{thm::CV_uncensored}.
However, instead of simply selecting one candidate estimator in the library, it is also possible to consider the continuous super learner algorithm which fits a weighted combination of the candidate learners. During the $V$-fold cross-validation process, a new data matrix is created, where each row $i$ consists in the set of predictions obtained from $Z_i$ by every candidate learners, together with the true outcome $T^*_i\wedge \tau$. Outcomes are then regressed onto the predictions, using another algorithm chosen by the user. This puts weights on the candidate learners. 
Such a continuous super learner will perform asymptotically at least as well as the best candidate learner. This is a direct consequence of Theorem~\ref{thm::CV_uncensored} by considering any combination of the candidate learners as a candidate learner itself, see~\citet{van_der_laan_super_2007}.

\subsection{Right-censored data}\label{sec::super_learner_RC}

\subsubsection{Super learner on standard pseudo-observations} 

\begin{figure}[ht!]
    \centering
    \includegraphics[scale=0.6]{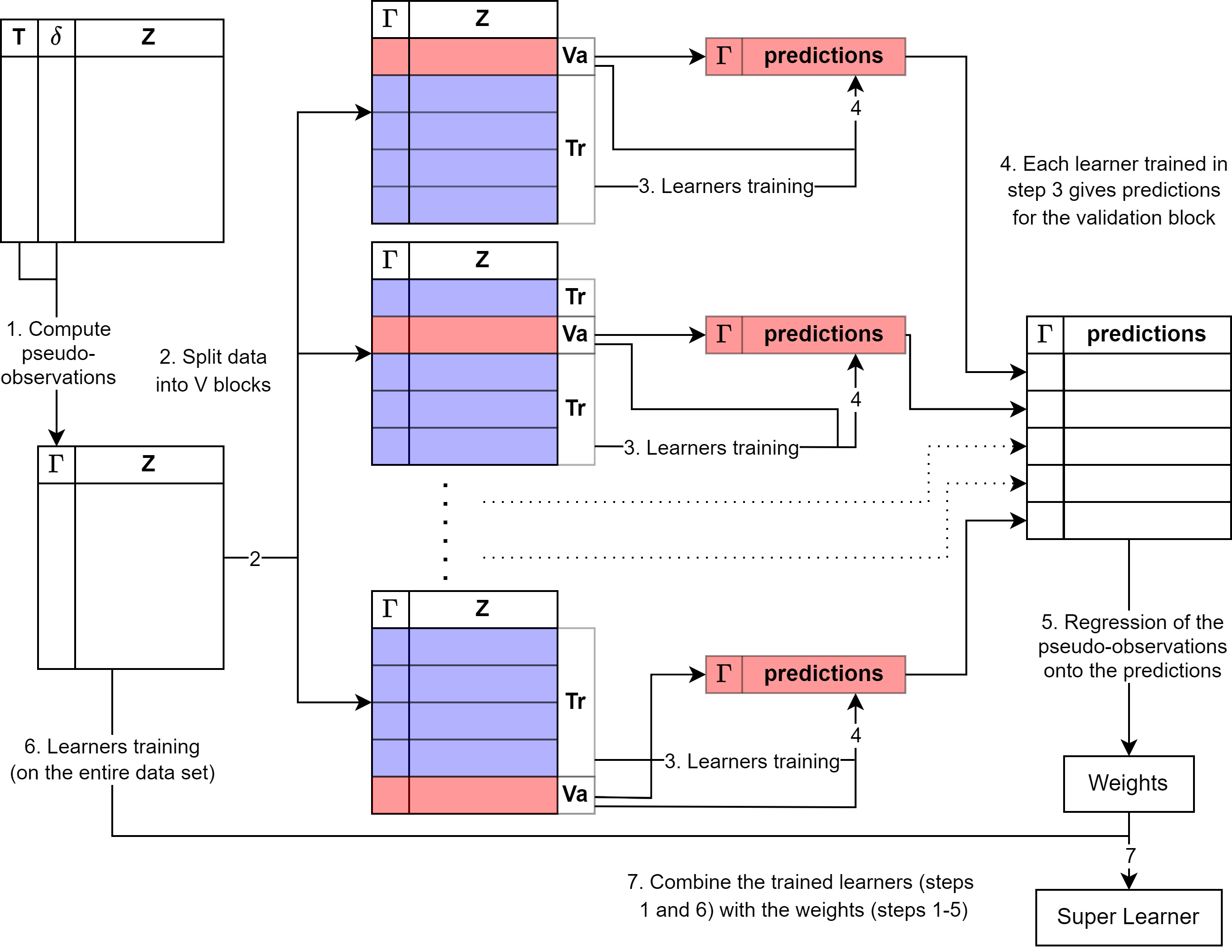}
    \caption{Diagram of the super learner based on standard pseudo-observations for right-censored data, see Equation~\eqref{eq::pobs}. Pseudo-observations are computed once and for all at the beginning.}
    \label{fig::pobs_SL}
\end{figure}

In order to take into account right-censored data, and motivated by the asymptotic result in Equation~\eqref{eq::pobs_jacobsen}, we propose a new method that consists in feeding pseudo-observations directly in the super learner, using the quadratic loss applied to those pseudo-observations. Our algorithm is therefore identical to the super learner described in~\citet{van_der_laan_super_2007}, with an additional first step for computing pseudo-observations. A diagram illustrating the method is provided in Figure~\ref{fig::pobs_SL}. The corresponding algorithm is detailed in Algorithm~\ref{algo::pobs_SL}. It describes both the discrete and the continuous pseudo-observations-based super learners, which differ from step 5 onwards.

\begin{algorithm}[H] 
\begin{algorithmic}
\REQUIRE Data $\Dn = \{O_i = (T_i,\delta_i,Z_i): i = 1,\dots,n\}$, number of folds $V \geq 2$, library of candidate algorithms for estimating the RMST based on pseudo-observations $\{\psihat_k : k = 1,\dots,K_n\}$.
\ENSURE A trained algorithm for the prediction of the restricted time to event.
\STATE \textbf{1.} Compute standard pseudo-observations for the whole dataset: $\{\Gamma_i: i=1,\dots,n\}$ (see Equation~\eqref{eq::pobs}).
\STATE \textbf{2.} Split the dataset into $V$ mutually exclusive blocks. 
\FOR{$v \in \{1, \dots, V\}$}
\STATE \textbf{3.} Divide observations according to the vector $B_n$ verifying, for all $i = 1,\dots,n$,
\[
    B_n(i) = \left\{
    \begin{array}{ll}
        1 & \text{if the $i$-th observation belongs to the $v$-th fold,} \\
        0 & \text{else,}
    \end{array}
    \right.
    \]
such that $\{(\Gamma_i,Z_i):  B_n(i)=0\}$ is the training set, of empirical law $\PnBn^0$, and $\{(\Gamma_i,Z_i):  B_n(i)=1\}$ is the validation set, of empirical law $\PnBn^1$.
Train all candidate learners $\{\psihat_k : k = 1,\dots,K_n\}$ on the training set, resulting in the trained predictors $\{\psihat_k(\PnBn^0) : k = 1,\dots,K_n\}$. 
\STATE \textbf{4.} Predict the restricted time to event for the data in the validation set with each trained candidate learner: $\{\psihat_{k,i} = \psihat_{k}(\PnBn^0)(Z_i): k = 1,\dots,K_n, B_n(i)=1\}$.
\ENDFOR
\STATE \textbf{Discrete super learner}
\STATE \textbf{5.} Identify the discrete super learner that minimizes the cross-validated risk:  $\khat = \argmin_{k \in \{1,\dots,K_n\}} \frac{1}{n}\sum_{i=1}^n (\Gamma_i - \psihat_{k,i})^2$.
\STATE \textbf{6.} Train the discrete super learner $\psihat_{\khat}$ on the entire dataset $\{(\Gamma_i,Z_i): i = 1,\dots,n\}$.
\STATE \textbf{Continuous super learner}
\STATE \textbf{5.} Regress the pseudo-observations $\{\Gamma_i: i = 1,\dots,n\}$ onto the predictions $\{\psihat_{k,i}: k = 1,\dots,K_n, i = 1,\dots,n\}$ with a parametric regression model in order to assign weights to the candidate learners.
\STATE \textbf{6.} Train all candidate learners $\{\psihat_k : k = 1,\dots,K_n\}$ on the entire dataset $\{(\Gamma_i,Z_i): i = 1,\dots,n\}$.
\STATE \textbf{7.} Combine the candidate learners trained on the whole dataset at step 6 with the weights obtained at step 5 to form the continuous super learner.
\end{algorithmic}
\caption{Pseudo-observations-based Super Learner}
\label{algo::pobs_SL}
\end{algorithm}

\subsubsection{Super learner on split pseudo-observations} 

We also propose a second algorithm based on split pseudo-observations. It requires to use a subset of the data, that we call the Kaplan-Meier (or KM) set, to compute the pseudo-observations in the validation set. A diagram illustrating this second method is provided in Figure~\ref{fig::pobs_split_SL}. The algorithm is detailed in Algorithm~\ref{algo::pobs_split_SL}, for both discrete and continuous super learners. The interest in this second algorithm lies in the conditional independence structure of the split pseudo-observations which allows to derive theoretical results similar to Theorem~\ref{thm::CV_uncensored}. Thus, the theoretical results below are provided with regard to the algorithm based on split pseudo-observations.

\begin{algorithm}[H] 
\begin{algorithmic}
\REQUIRE Data $\Dn = \{O_i = (T_i,\delta_i,Z_i): i = 1,\dots,n\}$, number of folds $V \geq 3$, library of candidate algorithms for estimating the RMST based on right-censored data $\{\psihat_k : k = 1,\dots,K_n\}$.
\ENSURE A trained algorithm for the prediction of the restricted time to event.
\STATE \textbf{1.} Split the dataset into $V$ mutually exclusive blocks.
\FOR{$v \in \{1, \dots, V\}$}
\STATE \textbf{2.} Divide observations according to the vector $B_n$ verifying, for all $i = 1,\dots,n$,
\[
    B_n(i) = \left\{
    \begin{array}{ll}
        2 & \text{if the $i$-th observation belongs to the $v$-th fold,} \\
        1 & \text{if the $i$-th observation belongs to the $(v+1)$-th fold (or first fold if $v=V$),} \\
        0 & \text{else,}        
    \end{array}
    \right.
    \]
such that $\{O_i:  B_n(i)=0\}$ is the training set, of empirical law $\PnBn^0$, $\{O_i:  B_n(i)=1\}$ is the KM set, of empirical law $\PnBn^1$, and $\{O_i:  B_n(i)=2\}$ is the validation set, of empirical law $\PnBn^2$.
Compute the split pseudo-observations for the data in the validation set, based on the KM set: $\{\Gamma_i^S = \Gamma_i(\PnBn^1): B_n(i)=2\}$ (see Equation~\eqref{eq::split_pobs}).
\STATE \textbf{3.} Train all candidate learners $\{\psihat_k : k = 1,\dots,K_n\}$ on the training set, resulting in the trained predictors $\{\psihat_k(\PnBn^0) : k = 1,\dots,K_n\}$.
\STATE \textbf{4.} Predict the time to event for the data in the validation set with each trained candidate learner: $\{\psihat_{k,i}=\psihat_{k}(\PnBn^0)(Z_i): k = 1,\dots,K_n, B_n(i)=1\}$.
\ENDFOR
\STATE \textbf{Discrete super learner}
\STATE \textbf{5.} Identify the discrete super learner that minimizes the cross-validated risk:  $\khat = \argmin_{k \in \{1,\dots,K_n\}} \frac{1}{n}\sum_{i=1}^n (\Gamma_i^S - \psihat_{k,i})^2$.
\STATE \textbf{6.} Train the discrete super learner $\psihat_{\khat}$ on the entire dataset $\{O_i: i = 1,\dots,n\}$.
\STATE \textbf{Continuous super learner}
\STATE \textbf{5.} Regress the split pseudo-observations $\{\Gamma_i^S: i = 1,\dots,n\}$ onto the predictions $\{\psihat_{k,i}: k = 1,\dots,K_n, i = 1,\dots,n\}$ with a parametric regression model in order to assign weights to the candidate learners.
\STATE \textbf{6.} Train all candidate learners $\{\psihat_k : k = 1,\dots,K_n\}$ on the entire dataset $\{O_i: i = 1,\dots,n\}$.
\STATE \textbf{7.} Combine the candidate learners trained on the whole dataset at step 6 with the weights obtained at step 5 to form the continuous super learner.
\end{algorithmic}
\caption{Split Pseudo-observations-based Super Learner}
\label{algo::pobs_split_SL}
\end{algorithm}

\begin{figure}
    \centering
    \includegraphics[scale=0.6]{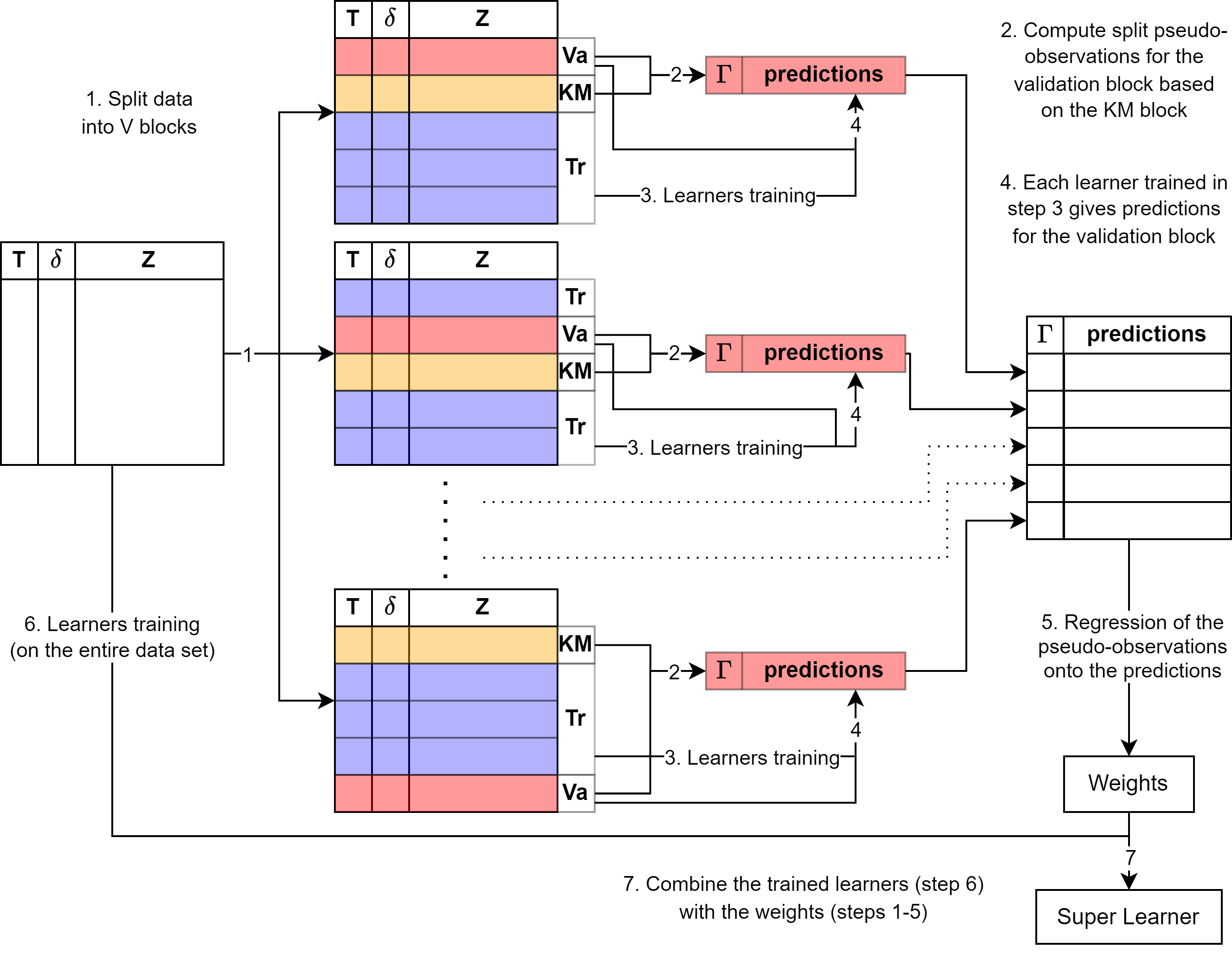}
    \caption{Diagram of the super learner based on split pseudo-observations for right-censored data, see Equation~\eqref{eq::split_pobs}. Pseudo-observations are computed during the cross-validation step, for each validation block, based on an additional subset of the data (KM).}
    \label{fig::pobs_split_SL}
\end{figure}

Formally, consider the censored observations $\{O_i = (T_i,\delta_i,Z_i), i=1,\dots,n\}$, i.i.d. of joint law $P \in \M$, with $T_i \in \Rbb^+$, $\delta_i \in \{0,1\}$, $Z_i \in \Rbb^d$.
As previously, $K_n$ candidate estimator mappings $\psihat_k$, $k=1,\dots,K_n$ are considered, among which we wish to select the best in terms of quadratic error, using cross-validation.
Combining split pseudo-observations with cross-validation imposes to divide the data in three subsets instead of two. Observations are divided according to an independent random vector $B_n = (B_n(i): i = 1,\dots,n) \in \{0,1,2\}^n$ into a first training set $\{O_i: i, B_n(i)=0\}$ of size $n_0 = n - \lfloor n p_{1,n}  \rfloor - \lfloor n p_{2,n}  \rfloor$ for $p_{1,n}, p_{2,n}, p_{1,n}+p_{2,n} \in (0,1)$, a second training set $\{O_i: i, B_n(i)=1\}$ of size $n_1 = \lfloor n p_{1,n}  \rfloor$ and a validation set $\{O_i: i, B_n(i)=2\}$ of size $n_2 = \lfloor n p_{2,n}  \rfloor$. The first training set is used to compute the candidate estimators of the RMST. The second training set is used to compute the Kaplan-Meier estimator which in turn is used for the computation of the pseudo-observations. We refer to this set as the Kaplan-Meier (KM) set. Pseudo-observations are computed for the data in the validation set. We denote $\PnBn^0$, $\PnBn^1$ and $\PnBn^2$ the empirical distributions of the three subsets. In this context, Equation~\eqref{eq::pobs_split_cond_mean_def} can be rewritten as
\begin{equation}\label{eq::pobs_split_cond_mean}
    \Esp[\Gamma_O(\PnBn^1) \mid Z,\PnBn^1,B_n] = \Esp[T^* \wedge \tau \mid Z] + \Esp[\xin \mid Z, \PnBn^1,B_n],
\end{equation}
where $\Gamma_O(\PnBn^1)$ is the split pseudo-observation constructed for the observation $O$ based on the distribution $\PnBn^1$ of the KM set.
It turns out that $\Esp[\xin \mid Z, \PnBn^1,B_n] = o_\Prob(1)$, so that the pseudo-observation plays a role of substitute to the true event time in the loss function. Consequently, the loss function also depends on the KM set used to compute the pseudo-observation,
\[
    L\po : (\psi, \PnBn^1, O) \mapsto L\bi(\psi,(\Gamma_O(\PnBn^1),Z)\bi) \in \Rbb, \text{ for } \psi \in \Psi, \PnBn^1 \sim \Pni, O \sim P.
\]

As above, we can define the risk of the parameter $\psi$ for distributions $P$ and $\PnBn^1$,
\[
    \Theta\po(\psi,\PnBn^1,P) = \EspBn \int L\po(\psi,\PnBn^1,o) dP(o).
\]
The risk minimizer now also depends on 
the empirical distribution $\PnBn^1$,
\begin{equation*}
    \psi^*_1 = \psi^*_1(\PnBn^1,P) = \argmin_{\psi \in \Psi} \Theta\po(\psi,\PnBn^1,P) = \argmin_{\psi \in \Psi} \EspBn \int L\po(\psi, \PnBn^1, o)dP(o),
\end{equation*}
and the corresponding optimal risk can be written as
\begin{equation}\label{eq::optimal_risk_pobs}
    \theta^*_1 = \Theta\po(\psi^*_1,\PnBn^1,P) = \min_{\psi \in \Psi} \Theta\po(\psi,\PnBn^1,P) = \min_{\psi \in \Psi} \EspBn \int L\po(\psi,\PnBn^1,o) dP(o).
\end{equation}
We can rewrite the cross-validated risk estimator,
\begin{align*}
    \thetahatpo(k) & =  \Theta\po\bi(\psihat_k(\PnBn^0),\PnBn^1,\PnBn^2\bi) \\
    & = \EspBn \int L\po\bi(\psihat_k(\PnBn^0),\PnBn^1,o\bi) d\PnBn^2(o) \\
    & = \EspBn \frac{1}{n_2} \sum_{i: B_n(i) =2} L\po\bi(\psihat_k(\PnBn^0),\PnBn^1,O_i\bi),
\end{align*}
the cross-validated selector,
\begin{equation}\label{eq::cv_selector_pobs}
    \khat\po = \argmin_{k \in \{1,\dots,K_n\}} \thetahatpo(k),
\end{equation}
the cross-validated conditional risk,
\begin{equation*}
    \thetatildepo(k) =  \Theta\po\bi(\psihat_k(\PnBn^0),\PnBn^1,P\bi) = \EspBn \int L\po\bi(\psihat_k(\PnBn^0),\PnBn^1,o\bi)dP(o),
\end{equation*}
and the cross-validated oracle selector
\begin{equation*}
    \ktilde\po = \argmin_{k \in \{1,\dots,K_n\}} \thetatildepo(k).
\end{equation*}
We compare the risk differences $\thetatildepo(\khat\po) - \theta^*_1$ and $\thetatildepo(\ktilde\po) - \theta^*_1$ to demonstrate optimality in a similar manner as in Theorem~\ref{thm::CV_uncensored}. This result is stated in Theorem~\ref{thm::CV_censored_pobs} below.

\begin{theorem} \label{thm::CV_censored_pobs}
    Let $O_1, \dots, O_n$ be a random sample from a data generating distribution $P$, where each $O_i = (T_i,\delta_i,Z_i)$ consists of three components, a univariate outcome $T_i \in \Rbb^+$, a binary censoring indicator $\delta_i \in \{0,1\}$ and a $d$-dimensional covariate vector $Z_i \in \Rbb^d$. Let $\tau < \tau_H$. Let $\{\psihat_k : k = 1,\dots,K_n\}$ denote a sequence of $K_n$ candidate estimators for the RMST, $\Esp[T^* \wedge \tau \mid Z]$, which is the risk minimizer for the quadratic loss function $L\po(\psi,\PnBn^1,O) = \bi(\Gamma_O(\PnBn^1) - \psi(Z)\bi)^2$, up to an asymptotically negligible term (see Equation~\eqref{eq::pobs_split_cond_mean}).
    Suppose there exists M such that $\tau \leq M < \infty$, and
    \begin{equation}\label{eq::assumption_thm_2}
        \lvert \Gamma_O(\PnBn^1) \rvert \leq M \text{ and } \sup_{Z \in \mathcal{Z},\psi \in \Psi} \lvert \psi(Z) \rvert \leq M \text{ almost surely.}
    \end{equation}
    Suppose that Assumption~\ref{ass::indep_cens} holds. 
\begin{description}
\item    \textbf{Finite sample result.} Let $M_1 = 8 M^2$, $M_2 = 16 M^2$ and $c(M, \gamma) = 2(1+\gamma)^2 (M_1/3 + M_2/\gamma)$. For all $\gamma > 0$, 
    \begin{equation}\label{eq::fin_sam_pobs}
        0 \leq \Esp [\thetatildepo(\khat\po) - \theta^*_1] \leq (1+2 \gamma) \Esp[ \thetatildepo(\ktilde\po) - \theta^*_1] + 2 c(M,\gamma) \frac{1 + \log(K_n)}{n p_{2,n}}\cdot
    \end{equation}
   \item \textbf{Asymptotic results.} If 
    \[
        \frac{\log(K_n)}{n p_{2,n} \Esp[ \thetatildepo(\ktilde\po) - \theta^*_1]} \to 0 \text{ as } n \to \infty, 
    \]
    then 
    \begin{equation}\label{eq::thm2_asym1}
        \frac{\Esp[ \thetatildepo(\khat\po) - \theta^*_1]}{\Esp[ \thetatildepo(\ktilde\po) - \theta^*_1]} \to 1 \text{ as } n \to \infty.
    \end{equation}
    \newpage
    Similarly, if 
    \[
        \frac{\log(K_n)}{n p_{2,n} (\thetatildepo(\ktilde\po) - \theta^*_1)} \to 0 \text{ in probability as } n \to \infty,
    \]
    then
    \begin{equation}\label{eq::thm2_asym2}
        \frac{\thetatildepo(\khat\po) - \theta^*_1}{\thetatildepo(\ktilde\po) - \theta^*_1} \to 1 \text{ in probability as } n \to \infty.
    \end{equation}
    \end{description}
\end{theorem}

With the additional result from~\citet{jacobsen_note_2016} in Equation~\eqref{eq::pobs_jacobsen}, adapted to split pseudo-observations in Equation~\eqref{eq::pobs_split_cond_mean}, we can also compare the risk differences $\thetatilde(\khat\po)~-~\theta^*$ and $\thetatilde(\ktilde) - \theta^*$.
In other words, we can compare the cross-validated selector on pseudo\hyp{}observations~\eqref{eq::cv_selector_pobs} to the oracle selector on true event times~\eqref{eq::cv_oracle_selector}, in terms of risks on true event times, see~\eqref{eq::optimal_risk} and~\eqref{eq::cv_risk}. The asymptotic result is given as a difference in Corollary~\ref{cor::CV_censored_pobs} below. The result in the form of a ratio requires stronger assumptions and is provided in the Supplementary Material.

\begin{corollary} \label{cor::CV_censored_pobs}
    Same setup and assumptions as in Theorem~\ref{thm::CV_censored_pobs}. 
    \begin{description}
\item    \textbf{Finite sample result.} Let $M_1 = 8 M^2$, $M_2 = 16 M^2$ and $c(M, \gamma) = 2(1+\gamma)^2 (M_1/3 + M_2/\gamma)$. For all $k=1,\dots,K_n$, let 
    \begin{align*}
        \phi_{n_1}(k) & = \EspBn\bii[\Esp\bii[\Esp\bi[\xin \mid Z , \PnBn^1,B_n\bi]^2 \\
        & \quad + 2\bi(\psi^*(Z) - \psihat_{k}(\PnBn^0)(Z)\bi)\Esp\bi[\xin \mid Z , \PnBn^1,B_n\bi] \midi P_{B_n},B_n \bii]\bii].
    \end{align*}
    Then, for all $\gamma > 0$, 
    \begin{align*}
        0 & \leq \Esp[\thetatilde(\khat\po) - \theta^*] \\
        & \leq (1+2 \gamma) \Esp[ \thetatilde(\ktilde\po) - \theta^*] + \Esp\bi[(1+2\gamma) \phi_{n_1}(\ktilde\po) - \phi_{n_1}(\khat\po) \bi] + 2 c(M,\gamma) \frac{1 + \log(K_n)}{n p_{2,n}}\cdot
    \end{align*}
\item    \textbf{Asymptotic results.} If 
    \[
        \frac{\log(K_n)}{n p_{2,n}} \to 0 \text{ as } n \to \infty,
    \]
    then
    \begin{equation}\label{eq::cor_asym1}
        \Esp[ \thetatilde(\khat\po) - \thetatilde(\ktilde) ] \to 0 \text{ as } n \to \infty,
    \end{equation}
    and
    \begin{equation}\label{eq::cor_asym2}
        \thetatilde(\khat\po) - \thetatilde(\ktilde) \to 0 \text{ in probability as } n \to \infty.
    \end{equation}
    \end{description}
\end{corollary}

\section{Simulations} \label{sec::simulations}

In this section, we conduct simulation studies to assess the validity of our approach on finite sample sizes. Our aim is threefold: first, we want to compare the values of the pseudo-observations constructed from the standard procedure and from our new procedure, the split pseudo-observations. Second, we want to compare the performance of our two proposed super learner algorithms, as described in Section~\ref{sec::super_learner_RC}, the one constructed from the standard pseudo-observations and the one constructed from the split pseudo-observations. Finally, we want to assess and compare the performances of our super learner algorithm based on standard pseudo-observations and two competitors: the Cox model and the Random Survival Forests (RSF) model~\citep[see][]{ishwaran_random_2008}. The simulations are implemented from two different Cox based models with increasing complexity. The two models were originally introduced in~\cite{cwiling_comprehensive_2023}. 

\begin{description}
    \item[Scheme 1:] Event times are simulated according to a Cox model with Weibull baseline hazard $\mathcal{W}(\nu,\kappa)$ and three covariates $Z = (Z^1,Z^2,Z^3)^{\top}$, where $Z^k \sim U[-a,a]$ for $k = 1,2,3$. Note that the survival function can be expressed as
    \[
        S(t\mid Z) = \exp\left[-\left(\frac{t}{\kappa}\right)^\nu\exp(\beta^{\top} Z)\right],
    \]
    with Cox regression parameters $\beta = (\beta_1,\beta_2,\beta_3)^{\top}$.
    Parameters are set to $\kappa = 2$, $\nu = 6$, $a = 5$, $(\beta_1, \beta_2, \beta_3) = (2,1,0)$. Censoring is simulated independently according to an exponential law with parameter $\lambda = 0.3$, leading to $47 \%$ censored data. The time horizon $\tau$ is chosen as the $90$th percentile of the observed times $T$ which corresponds to $\tau=3.6$ in this setting.
    \item[Scheme 2:] Event times are simulated according to a Cox model with Weibull baseline hazard $\mathcal{W}(\nu,\kappa)$, $\kappa = 2$, $\nu = 6$, this time with $\lambda(t\mid Z)=\lambda_0(t)\exp(g(Z))$ and 
    \begin{align*}
        g(Z) & =  Z^3 - 3 Z^5 + 2 Z^1 Z^{10} + 4 Z^2 Z^7 + 3 Z^4 Z^5 - 5 Z^6 Z^{10} \\
        & \quad+ 3 Z^8 Z^9 + Z^1 Z^4 - 2 Z^6 Z^9 - 4 Z^3 Z^4 - Z^7 Z^8,
    \end{align*}
    so that the survival function is expressed as
    \begin{align*}
        S(t\mid Z) & = \exp\left[-\left(\frac{t}{k}\right)^\nu\exp(g(Z))\right].
    \end{align*}
    Let $Z = (Z^1,\dots,Z^{15})^{\top}$. We simulate the covariates $Z = (Z^1,\dots,Z^{15})^{\top}$ with $Z^j \sim \mathcal B(0.4)$ for $j \in \{2,4,6,9,11,12\} $ and $Z^j \sim U[0,1],$ for $j \in \{1,3,5,7,8,10,13,14,15\}$. As a result, only the first $10$ covariates are associated with the event times. However, the other $5$ covariates, that are independent from the event times, will still be included in our regression models when analyzing those data. The censoring distribution is the same as in scheme 1, leading to $47 \%$ censored data. The time horizon $\tau$ is chosen as the $90$th percentile of the observed times which corresponds to $\tau=2.8$ in this setting.
\end{description}

\subsection{Standard versus split pseudo-observations}

\begin{figure}[ht!]
    \centering
    \includegraphics[scale=0.75]{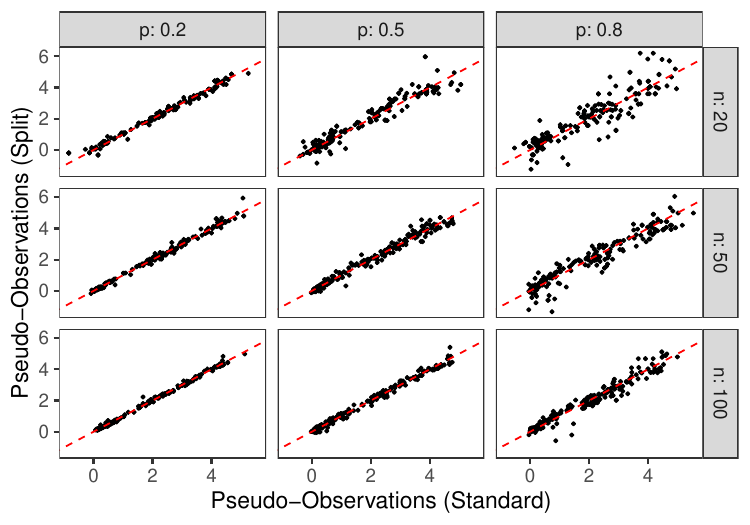}
     \caption{Comparison of standard and split pseudo-observations on data simulated with the first simulation scheme, with varying sample sizes $n$ and $n_2 = \rho n$ for the split pseudo-observations.}
    \label{fig::standard_vs_split}
\end{figure}

In this section, we compare the values of the standard pseudo-observations and the split pseudo-observations constructed as described in Section~\ref{sec::pobs}. We construct samples of different sizes $n\in\{20, 50, 100\}$ according to the simulation scheme 1 and we vary the proportion of the two samples $D_{n_1}$ and $D_{n_2}$ used for the construction of the split pseudo-observations by setting $n_2 = \rho n$ with $\rho\in\{0.2,0.5,0.8\}$. The process is repeated $50$ times. For each split pseudo-observation, a corresponding standard pseudo-observation can also be computed and compared. For clarity, $200$ pseudo-observations are randomly drawn for each scenario and their scatter plots are displayed in Figure~\ref{fig::standard_vs_split}.
Overall, we observe a very close agreement between the values of the standard and split pseudo-observations. As the sample size increases, the similarity between the two methods increases even more. A larger discrepancy occurs when $n_2$ is larger. This is a consequence of the sample size $n_1$ used for the Kaplan-Meier estimator in the construction of the split pseudo-observations which is smaller when $n_2$ is larger. However, even in the worst scenario ($\rho=0.8$ and $n=20$), the discrepancy between the two methods remains moderate.

\begin{figure}
    \centering
    
    \subfloat[Simulation scheme 1]{%
        \includegraphics[width=0.89\linewidth]{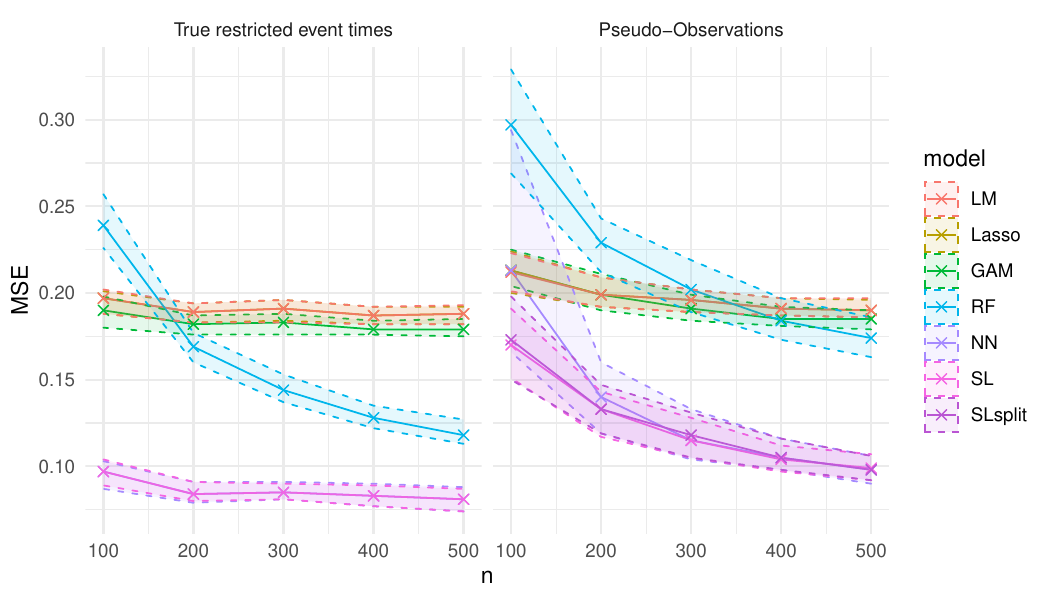}%
        \label{fig::mse_graph_B}%
    }

    \subfloat[Simulation scheme 2]{%
        \includegraphics[width=0.89\linewidth]{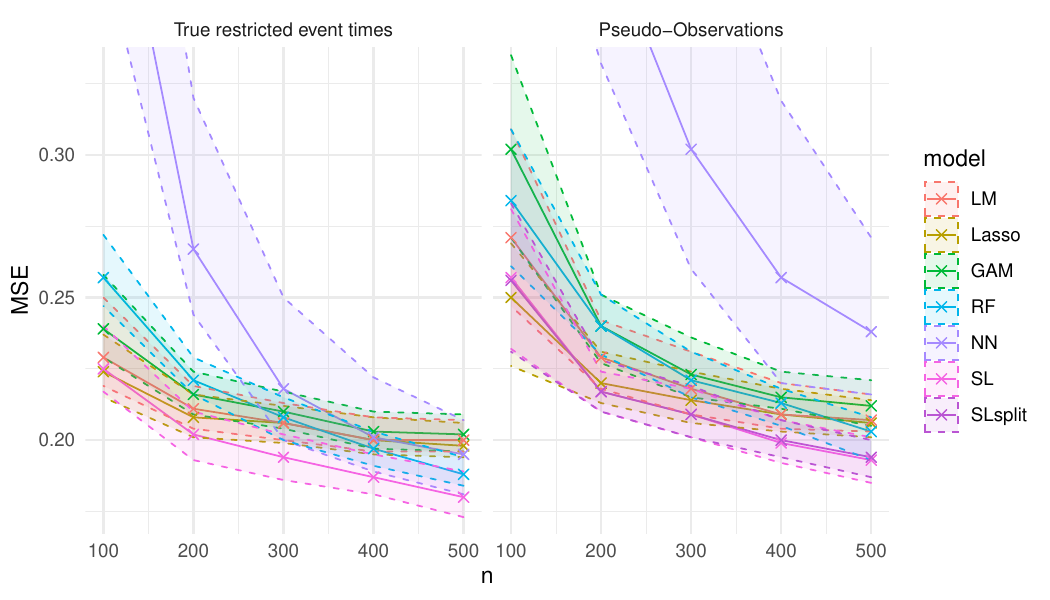}%
        \label{fig::mse_graph_C}%
    }

    \caption{Mean Squared Errors (MSE) for the restricted time prediction with algorithms implemented on the true uncensored observations (first column) and the pseudo-observations (second column). The considered algorithms are: the linear model (LM), the Lasso, the generalized additive model (GAM), the random forests (RF), the neural network (NN) and the super learner (SL) combining all these algorithms. In the second column, standard and split pseudo-observations-based super learners are implemented (SL and SLsplit). Their curves overlap almost entirely. The sample size of the training set ranges from $n=100$ to $n=500$ and two simulation schemes are considered in the top and bottom rows. The MSEs of all the algorithms are computed on an independent test set of size $1000$ and the whole procedure is repeated $80$ times. Values of the median (cross), first and third quantiles (dashed lines) of the MSEs are reported.}
    \label{fig::mse_graph}
\end{figure}

\subsection{Super learner implemented with standard versus split pseudo-observations}

In this section, we want to compare the implementation of the two super learner algorithms, as introduced in Section~\ref{sec::super_learner_RC}: The first one is based on standard pseudo-observations and the second one is based on the split pseudo-observations. While the two types of pseudo-observations are very similar (as illustrated in the previous section), the two algorithms are still different in their implementations. The first one applies candidate learners directly on the pseudo-observations, treating those observations as uncensored data.
On the contrary, in the second approach, the candidate learners are applied to the censored observations and require algorithms that can handle right-censored observations. In order to be able to compare the two approaches, we only consider learning algorithms based on pseudo-observations. More precisely, for the second approach, pseudo-observations are computed from the training set before applying the learning algorithms. The algorithms considered for the two methods are: the Linear Model (LM), the Lasso, the Generalized Additive Model (GAM), the Random Forests (RF) and the Neural Network (NN). For both methods, once the weights attributed to the candidate learners are computed, the algorithms are trained one last time on the whole dataset and combined using those weights, see Figure~\ref{fig::pobs_SL} and Figure~\ref{fig::pobs_split_SL}. In our simulation setting, this last step is identical for both methods, i.e. standard pseudo-observations are computed on the whole dataset and used to train LM, Lasso, GAM, RF and NN. The only difference between the two approaches results in the computation of the weights in the construction of the super learner. The performances of the two methods are further compared with the oracle method that applies the same algorithms directly to the unoberved true event times. All super learners are trained using $6$-fold cross-validation. The results are displayed in Figure~\ref{fig::mse_graph}, for the three different methods in the two simulation schemes with increasing sample sizes $n\in\{100,200,300,400,500\}$. The MSEs of the super learner algorithms and of every candidate learners are computed on an independent test set of size $1000$ and the whole procedure is repeated $80$ times for each sample size. The median and the first and third quantiles of the MSEs are reported in the figure. Exact values can be found in Table~1 in the Supplementary Material.

We first observe that the two pseudo-observations-based super learner algorithms share similar performances for all sample sizes, regardless of the simulation scheme.
This indicates that the two algorithms are almost identical in practice. Also, their performance is only slightly worse than the oracle algorithm. For the three algorithms, the super learner outperforms all individual learners. This is an illustration of Theorem~\ref{thm::CV_uncensored} in~\citet{van_der_laan_super_2007} for the oracle algorithm and of Corollary~\ref{cor::CV_censored_pobs} for the split pseudo-observations algorithm. While we were not able to derive a similar theoretical result for the super learner based on standard pseudo-observations, the simulations clearly support this conjecture. Finally, the distribution of the weights assigned to each learner can be found in the Supplementary Material in Figure~1. We observe in the simulation scheme 1 that the largest weight is assigned to the NN algorithm with an increasing value of the weight when the sample size increases. For instance, for $n=500$, the median weight allocated to the NN algorithm is equal to $0.906$, $0.762$ and $0.727$ for the oracle algorithm, the super learner combined with standard pseudo-observations and the super learner combined with split pseudo-observations, respectively. This is in accordance with the results in Figure~\ref{fig::mse_graph} where the NN algorithm outperforms all other individual algorithms and is only slightly less performing than the super learner. In the simulation scheme 2, the weights are more equally shared among all learners. The RF tend to have the largest weights with an increasing value when the sample size increases. For example, when $n=500$, the median weights for the RF, LM and NN are equal to $0.525$, $0.232$, $0.172$ respectively for the oracle algorithm, equal to $0.467$, $0.345$, $0.071$ respectively for the super learner combined with standard pseudo-observations and equal to $0.462$, $0.370$, $0.014$ respectively for the super learner combined with split pseudo-observations.

Since the two implementations of the super learner are almost identical, we only study the super learner based on standard pseudo-observations in the next section. This algorithm has the advantage to have a lower computational complexity and to allocate more data to the training set.

\subsection{Prediction performance of the super learner and comparison with other methods}

\begin{figure}[ht!]
    \centering
    
    \subfloat[Simulation scheme 1]{%
        \includegraphics[width=0.49\linewidth]{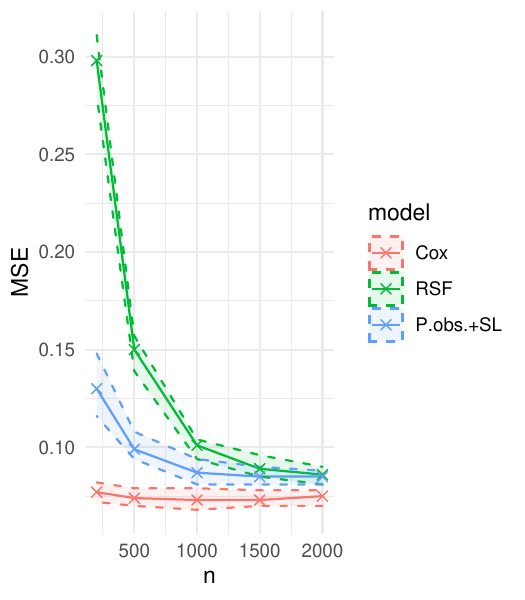}%
        \label{fig::mse_surv_graph_B}%
    }
    \hfill
    \subfloat[Simulation scheme 2]{%
        \includegraphics[width=0.49\linewidth]{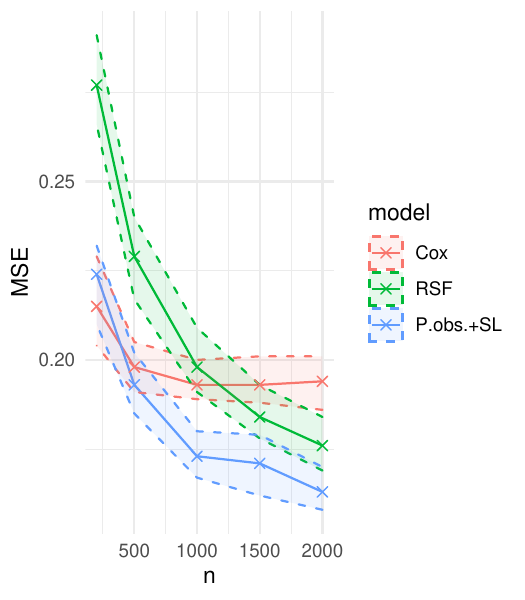}%
        \label{fig::mse_surv_graph_C}%
    }

    \caption{Mean Squared Errors (MSE) for the restricted time prediction with three different methods: the Cox model (with no interactions between covariates), the RSF and the super learner combined with standard pseudo-observations. The algorithms are trained on samples of size $n\in\{200,500,1000,1500,2000\}$ and two simulation schemes are considered in the left and right columns. The MSEs are computed on an independent test set of size $1000$ and the whole procedure is repeated $80$ times. Values of the median (cross), first and third quantiles (dashed lines) of the MSEs are reported.}
    \label{fig::mse_surv_graph}
\end{figure}

In this section, we compare our super learner algorithm combined with the standard pseudo-observations to the standard Cox model (without including interactions) and the Random Survival Forests (RSF) algorithm~\citep[see][]{ishwaran_random_2008}. For the super learner, we employ the same candidate learners as in the previous section and still use $6$-fold cross-validation. For the Cox and RSF algorithms, the RMST is obtained by integrating the estimated survival curve between $0$ and $\tau$. 
We simulated data of increasing size $n\in\{200,500,1000,1500,2000\}$ in the two simulation schemes. As previously, the MSEs of all prediction models are computed on an independent test set of size $1000$ and the procedure is repeated $80$ times for each sample size. The median, first and third quantiles of the MSEs are reported in Figure~\ref{fig::mse_surv_graph}. 

The first simulation scheme consists in simulating time to events according to a standard Cox model without interactions between covariates. Hence, it is not surprising that Figure~\ref{fig::mse_surv_graph_B} shows a better performance for the Cox model compared to the other methods. However, the RSF and our method perform only slightly worse than the Cox model, with an increasing performance as the sample size increases. For small sample sizes, the super learner is clearly more performing than the RSF. On the other hand, the second simulation scheme is a Cox model with complex interactions between the covariates. In the results presented in Figure~\ref{fig::mse_surv_graph_C}, our method outperforms the Cox model for $n\geq 500$ and also exhibits better performances than the RSF for all sample sizes.

\section{Applications on real data} \label{sec::realdata}

\subsection{The maintenance dataset}

We first study a maintenance dataset publicly available in the \texttt{PySurvival} Python package. It gathers data on $1000$ machines, for which the number of weeks in activity has been recorded. The aim is to study the failure time of the machines. Overall
$40\%$ of the collected times are right-censored, and no failures occurred before $60$ weeks of operation. Observed times range from $1$ to $93$ weeks, and we set the time horizon to the $90$th quantile of the observed times which corresponds to $\tau = 88$. The dataset contains $5$ predictors: three continuous covariates, pressure, moisture, temperature and two categorical variables, the team using the machine (three levels) and the manufacturer (four levels). 

We first investigate the possible dependence between censoring and covariates. To that aim we fit a Cox model and a RSF to the data using the censoring time as the outcome. Both methods indicate no evidence of dependence: the global p-value for the Cox model is equal to $0.604$ while the variable importance measure per permutation in the RSF does not exceed $0.005$ for any of the variables (see Table~2 in Supplementary Material). 

Next, we aim at predicting the time to failure using several models. We were not able to implement the Cox model from the \texttt{coxph} function because the Newton-Raphson algorithm ran out of iterations without converging. We use instead the RSF and the super learner based on pseudo-observations. For the latter, cross-validation is fixed to $6$ folds and the library was composed of LM, Lasso, GAM and RF. Both learning models are compared and analyzed using the methods developed in~\citet{cwiling_comprehensive_2023}. 

First, the MSE is evaluated with the weighted residual sum of squares (WRSS), an inverse probability censoring weights (IPCW) estimator derived from the Brier score. Since our previous analysis supports the independent censoring assumption, we simply use the Kaplan-Meier estimator to compute the censoring weights. Results obtained on $10$-fold cross-validation are displayed in Figure~\ref{fig::mse_mtnc}, indicating reasonable prediction errors in both cases considering the time range, with better performances for our method. 

Next, a prediction interval for the true restricted time to failure of level $90\%$ is computed for each item. These intervals are computed with the IPCW Rank-One-Out Split Conformal algorithm. This method relies on the distribution of the weighted residuals where, again, the Kaplan-Meier estimator is used for the censoring weights. The median length of the prediction intervals equals $7.217$ and $6.789$ for the RSF and the super learner combined with pseudo-observations, respectively. An illustration of these intervals, for some machines chosen randomly in the dataset, is provided in the Supplementary Material in Figure~2. 

Finally, a global variable importance test is implemented based on the Leave-One-Covariate-Out conformal methodology. This test indicates whether removing a variable from the training set significantly deteriorates the prediction quality. This test is conducted by splitting the data into a training set (used for the implementation of the RMST estimator) and a validation set (used for the implementation of the test itself). Since this procedure is sensitive to the split, we conduct a multi-splitting procedure to stabilize the results along with a method for controlling the familywise rejection rate~\citep[see][for more details on this method]{diciccio_exact_2020}. Consequently, the test is implemented at the $10\%$ level on $M=40$ different splits of the data and the $M$ p-values are aggregated by computing twice the median value. This ensures to obtain a $10\%$ level of the overall test. The results are displayed in Table~\ref{tab::multisplitting_mtnc}. They show clear evidence of the influence of the team and provider variables on the failure time, while pressure does not alter the prediction quality and moisture seems to only have very limited impact in the RSF. However, both models do not reach the same conclusion for temperature, which seems to have a strong predictive power in the RSF but not in the super learner.

\begin{figure}
    \centering
    \includegraphics[scale=0.65]{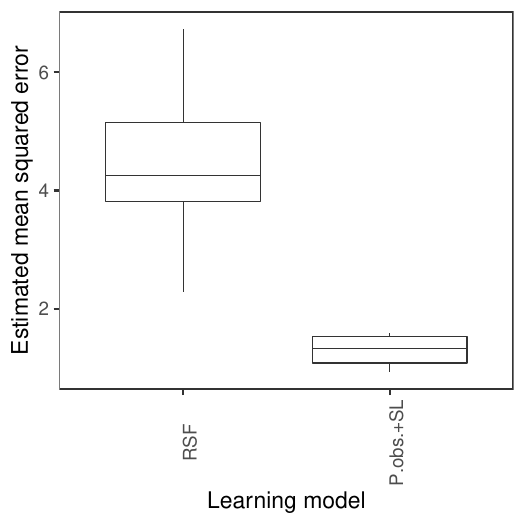}
    \caption{Estimation of the MSE on the \texttt{maintenance} dataset using the WRSS based on $10$-fold cross-validation, for the RSF and the super learner based on pseudo-observations.}
    \label{fig::mse_mtnc}
\end{figure}

\begin{table}[ht!]
    \centering
    \resizebox{0.55\textwidth}{!}{
    \begin{tabular}{|l|ll|ll|}
\hline
Variable & \multicolumn{2}{l|}{P-value RSF} & \multicolumn{2}{l|}{P-value P.obs.+SL} \\
\hline
Pressure & 1 &  & 1 &  \\

Moisture & 0.257 &  & 1 &   \\

Temperature & 0.003 & ** & 1 &  \\

Team & 0 & *** & 0 & \qquad ***  \\

Provider & 0 & *** & 0 & \qquad ***  \\
\hline
    \end{tabular}
}
    \smallskip
    
    Significance codes: 0 '***' 0.001 '**' 0.01 '*' 0.05 '.' 0.1 ' ' 1
    \caption{Variable importance of the \texttt{maintenance} dataset for the RSF and the super learner based on pseudo-observations using a multi-splitting procedure with $M=40$ splits. The global variable importance test is performed on each split and the $M$ p-values are aggregated by computing twice the median value in order to obtain a $10\%$ overall test.}
    \label{tab::multisplitting_mtnc}
\end{table}

\subsection{The colon cancer dataset}

\begin{figure}[ht!]
    \centering
    \includegraphics[scale=0.65]{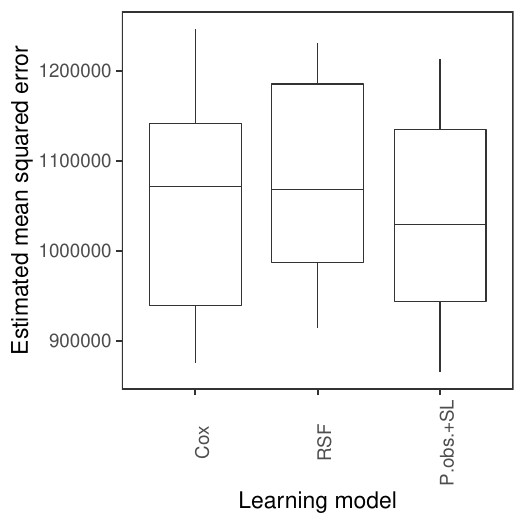}
    \caption{Estimation of the MSE on the \texttt{colon} dataset using the WRSS based on $10$-fold cross-validation, for the Cox model, the RSF and the super learner based on pseudo-observations.}
    \label{fig::mse_colon}
\end{figure}

We study now a dataset from a clinical trial on the effect of chemotherapy for colon cancer. The \texttt{colon} data are available in the R \texttt{survival} package. The complete case analysis contains $888$ patients whose follow-up times range from $8$ days to $3329$ days ($9.12$ years). The goal of the study is to analyze the progression free survival (PFS) defined as the time elapsed from randomization to the minimum between recurrence and death. For this endpoint, $46\%$ of the observations are censored, that is neither recurrence or death is observed. We set the time horizon to the $90$th quantile of the observed times which corresponds to $\tau = 2672$ days. We use as predictors the type of treatment administrated (three levels), sex (binary), age (in years), obstruction indicator of the colon by the tumor (binary), whether the colon was perforated or not (binary), whether or not it adhered to nearby organs (binary), number of lymph nodes with detectable cancer (integer value that ranges from $0$ to $33$), level of differentiation of the tumor (three levels), extent of local spread (four levels), and whether the time from surgery to registration was short or long (binary). 

We first investigate the possible dependence between censoring and covariates. As for the previous dataset, we fit a Cox model and a RSF to the data using the censoring time as the outcome. Both methods indicate no evidence of dependence: the global p-value for the Cox model is equal to $0.383$ while the variable importance measure per permutation in the RSF does not exceed $0.015$ for any of the variables (see Table~3 in Supplementary Material). 

Next, we predict the time to progression with the Cox model, the RSF and the pseudo-observations-based super learner. For the latter, $6$-fold cross-validation is implemented with the library composed of LM, Lasso, GAM and RF. As previously, the performances of the learning models are analyzed using the methods developed in~\citet{cwiling_comprehensive_2023} and the censoring distribution is estimated using the Kaplan-Meier estimator for all the methods involving IPCW. 

In terms of MSE, all three models seem to have similar performances (see Figure~\ref{fig::mse_colon}). The prediction error is particularly large, emphasizing the difficulty for all the learners to provide accurate individual predictions. 
This is in accordance with the length of the $90\%$ prediction intervals, displayed in Figure~3 in Supplementary Material, whose median equals $2961$ for the Cox model, $3098$ for the RSF and $3050$ for the super learner combined with pseudo-observations. 
Finally, the variable importance test is implemented based on the multi-split approach with $40$ different splits and overall rejection rate equal to $10\%$. The results are presented in Table~\ref{tab::multisplitting_colon}. We observe that the importance of the type of treatment and extent of local spread variables are highly significant for the Cox and super learner algorithms while moderately significant for the RSF (p-value equals $0.084$ and $0.063$, respectively). The importance of all the other variables is far from significance for all algorithms.

\begin{table}[ht!]
    \centering
    \resizebox{0.7\textwidth}{!}{
    \begin{tabular}{|l|ll|ll|ll|}
\hline
Variable & \multicolumn{2}{l|}{P-value Cox} & \multicolumn{2}{l|}{P-value RSF} & \multicolumn{2}{l|}{P-value P.obs.+SL} \\
\hline
Treatment & 0 & *** & 0.084 & . & 0.002 & ** \\

Sex & 0.981 & & 0.724 &  & 0.622 & \\

Age & 1 & & 1 &  & 1 & \\

Obstruction & 1 & & 1 &  & 1 & \\

Perforation & 1 & & 1 &  & 1 & \\

Adherence & 1 & & 1 &  & 1 & \\

Nodes & 1 & & 0.637 &  & 0.662 & \\

Differentiation & 1 & & 1 &  & 1 & \\

Spread & 0 & *** & 0.063 & . & 0.002 & ** \\

Surgery & 1 & & 1 &  & 1 & \\
\hline
    \end{tabular}
}
    \smallskip
    
    Significance codes: 0 '***' 0.001 '**' 0.01 '*' 0.05 '.' 0.1 ' ' 1
    \caption{Variable importance of the \texttt{colon} dataset for the Cox model, the RSF and the super learner based on pseudo-observations using a multi-splitting procedure with $M=40$ splits. The global variable importance test is performed on each split and the $M$ p-values are aggregated by computing twice the median value in order to obtain a $10\%$ overall test.}
    \label{tab::multisplitting_colon}
\end{table}

\section{Conclusion}

Predicting the restricted time to event is of great interest in many applications. When using a quadratic loss function, this problem amounts to estimating the RMST. In this work we focused on the super learner combined with pseudo-observations to achieve this goal. The algorithm is simple to implement as it only requires an extra initial step for constructing the pseudo-observations. From there, the standard super learner is simply applied to the pseudo-observations and allows to get access to all the prediction algorithms that are classically used when dealing with uncensored observations. It also enjoys the nice property of the super learner: It will automatically select the best learner among all the candidates (for the discrete super learner) or it will provide the best combination of the candidates (for the continuous super learner). This is an important feature in applications since the best candidate can vary from one study to another or when the sample size increases. Those results were proved in our context of right-censored data, when the split pseudo-observations are used. We also observed this property on simulated data and we empirically showed that combining the super learner with split or standard pseudo-observations provides very similar results. This is why we recommend to combine the super learner with standard pseudo-observations in practice, as it provides a simpler method and it allows to allocate more data to the training of candidate learners.  

Interestingly, the use of split pseudo-observations might go beyond the scope of the super learner when studying theoretical properties of an estimator based on pseudo-observations. Indeed, they are almost identical to standard pseudo-observations, yet they do not suffer from dependency problems which makes them much easier to study in theoretical problems. 

Our method relies on the independent censoring assumption inherent to pseudo\hyp{}observations. When dealing with categorical covariates it is still possible to use the approach of~\cite{andersen_pseudo-observations_2010} (see Section 2.2) based on a mixture of Kaplan-Meier estimators to relax this assumption. Finally, our method could easily be extended to other situations such as recurrent events, competing risks, or left truncation by using a different definition of pseudo-observations adapted to these types of data~\citep[see][]{binder_pseudo-observations_2014,grand_note_2019,furberg_bivariate_2023}. This is left to future research work.

In practice, several methods exist to implement pseudo-observations-based super learners. It would be of interest to provide some guidelines to practitioners, for instance on the choice between our approach and the ones based on IPCW or AUC losses. The main novelty of our work, as compared to previous methods, relies on theoretical results about the pseudo-observations-based super learner. Empirical studies comparing the different methods are warranted.

\section*{Acknowledgments}

We thank Professor Antoine Chambaz for his helpful suggestions regarding the application of the super learner to pseudo-observations. His insights greatly improved the quality of our work.

\section{Technical Results} \label{sec::proofs}

\begin{lemma}\label{lemma::risk_minimizer_quad_pobs}
     If we consider the quadratic loss function $L\po(\psi,\PnBn^1,O) = \bi(\Gamma_O(\PnBn^1) - \psi(Z)\bi)^2$, with $O$ independent from $\PnBn^1$ and $B_n$, then the risk minimizer is $\psi^*_1(Z) = \Esp\bi[\Gamma_O(\PnBn^1) \mid Z, \PnBn^1,B_n\bi]$.
\end{lemma}
\begin{proof}
    We have
    \begin{align*}
        & \Esp\bii[L\po\bi(\psi,\PnBn^1,O\bi) \midi \PnBn^1,B_n \bii] \\
        & \quad = \Esp\biii[\bii(\Gamma_O(\PnBn^1) - \psi(Z)\bii)^2 \midii \PnBn^1,B_n \biii] \\
        & \quad = \Esp\biii[\bii(\Gamma_O(\PnBn^1) - \Esp\bi[\Gamma_O(\PnBn^1) \mid Z, \PnBn^1,B_n\bi] + \Esp\bi[\Gamma_O(\PnBn^1) \mid Z, \PnBn^1,B_n \bi] - \psi(Z)\bii)^2 \midii \PnBn^1,B_n \biii] \\
        & \quad = \Esp\biii[\bii(\Gamma_O(\PnBn^1) - \Esp\bi[\Gamma_O(\PnBn^1) \mid Z, \PnBn^1,B_n\bi]\bii)^2 \midii \PnBn^1,B_n \biiii] \\
        & \qquad + \Esp\biii[\bii(\Esp\bi[\Gamma_O(\PnBn^1) \mid Z, \PnBn^1,B_n \bi] - \psi(Z)\bii)^2 \midii \PnBn^1,B_n \biii] \\
        & \qquad + 2 \Esp\biii[\bii(\Gamma_O(\PnBn^1) - \Esp\bi[\Gamma_O(\PnBn^1) \mid Z, \PnBn^1,B_n\bi]\bii)\bii(\Esp\bi[\Gamma_O(\PnBn^1) \mid Z,B_n\bi] - \psi(Z)\bii) \midii \PnBn^1,B_n \biii]  
    \end{align*}
    and
    \begin{align*}
        & \Esp\biii[\bii(\Gamma_O(\PnBn^1) - \Esp\bi[\Gamma_O(\PnBn^1) \mid Z, \PnBn^1,B_n\bi]\bii) \bii(\Esp\bi[\Gamma_O(\PnBn^1) \mid Z, \PnBn^1,B_n\bi] - \psi(Z)\bii) \midii \PnBn^1,B_n \biii] \\
        & \quad = \Esp\biii[\Esp\biii[\bii(\Gamma_O(\PnBn^1) - \Esp\bi[\Gamma_O(\PnBn^1) \mid Z, \PnBn^1,B_n\bi]\bii) \\
        & \qquad \bii(\Esp\bi[\Gamma_O(\PnBn^1) \mid Z, \PnBn^1,B_n\bi] - \psi(Z)\bii)\midii Z,  \PnBn^1,B_n\biii] \midii \PnBn^1,B_n \biii] \\
        & \quad = \Esp\biii[\bii(\Esp\bi[\Gamma_O(\PnBn^1) \mid Z, \PnBn^1,B_n\bi] - \Esp\bi[\Gamma_O(\PnBn^1) \mid Z, \PnBn^1,B_n\bi]\bii) \\
        & \qquad \bii(\Esp\bi[\Gamma_O(\PnBn^1) \mid Z, \PnBn^1,B_n \bi] - \psi(Z)\bii) \midii \PnBn^1,B_n \biii] \\
        & \quad = 0.
    \end{align*}
    Hence
    \begin{align*}
        & \argmin_{\psi \in \Psi} \EspBn \Esp\bi[L\po(\psi,\PnBn^1,O) \mid \PnBn^1,B_n\bi] \\
        & \quad = \argmin_{\psi \in \Psi} \EspBn\Esp\biii[\bii(\Esp\bi[\Gamma_O(\PnBn^1) \mid Z, \PnBn^1,B_n\bi] - \psi(Z)\bii)^2 \midii \PnBn^1,B_n\biii].
    \end{align*}
\end{proof}

We prove that Lemma 3 in~\citet{dudoit_asymptotics_2005} is still valid with the pseudo-observation quadratic loss function.
\begin{lemma}\label{lemma::quadratic_var_bound_pobs}
    Same setup and assumptions as in Theorem~\ref{thm::CV_censored_pobs}. Conditional on the training sets empirical distributions $\PnBn^0$, $\PnBn^1$ and split vector $B_n$, define, for any $k=1,\dots,K_n$, the random variable $X_k = L\po(\psihat_k(\PnBn^0),\PnBn^1,O) - L\po(\psi^*_1,\PnBn^1,O)$, where $O$ is independent from $\PnBn^0$, $\PnBn^1$, $B_n$, and $L\po$ denotes the quadratic loss function $L\po(\psi,\PnBn^1,O) = (\Gamma_O(\PnBn^1) - \psi(Z))^2$. Then
    \begin{equation*}
        \Var\bi(X_k \mid \PnBn^0, \PnBn^1, B_n\bi) \leq M_2 \Esp\bi[X_k \mid \PnBn^0, \PnBn^1, B_n\bi].
    \end{equation*}
\end{lemma}
\begin{proof}
    For the quadratic loss function, $\psi^*_1(Z) = \Esp\bi[\Gamma_O(\PnBn^1) \mid Z, \PnBn^1,B_n\bi]$ from Lemma~\ref{lemma::risk_minimizer_quad_pobs}. As $$X_k = \bi(\psi^*_1(Z) - \psihat_k(\PnBn^0)(Z)\bi)\bi(2\Gamma_O(\PnBn^1) - \psihat_k(\PnBn^0)(Z)-\psi^*_1(Z)\bi),$$
    we have
    \begin{align*}
        & \Esp\bi[X_k \mid \PnBn^0, \PnBn^1, B_n\bi] \\
        & = \Esp\bi[\Esp\bi[X_k \mid Z, \PnBn^0, \PnBn^1, B_n\bi]\mid \PnBn^0, \PnBn^1, B_n\bi] \\
        & = \Esp\biii[\bii(\psi^*_1(Z) - \psihat_k(\PnBn^0)(Z)\bii) \\
        & \qquad\bii(2\Esp\bi[\Gamma_O(\PnBn^1) \mid Z, \PnBn^1,B_n\bi]  - \psihat_k(\PnBn^0)(Z)-\psi^*_1(Z)\bii) \midii \PnBn^0, \PnBn^1, B_n\biii] \\
        & = \Esp\bii[\bi(\psi^*_1(Z) - \psihat_k(\PnBn^0)(Z)\bi)^2 \midi \PnBn^0, \PnBn^1, B_n\bii].
    \end{align*}
    As $\bi\lvert 2\Gamma_O(\PnBn^1) - \psihat_k(\PnBn^0)(Z)-\psi^*_1(Z) \bi\rvert \leq 4M$ by assumption~\eqref{eq::assumption_thm_2}, and with $M_2 = (4M)^2$, we find
    \begin{align*}
        \Var\bi(X_k \mid \PnBn^0, \PnBn^1, B_n\bi) & \leq \Esp\bi[X_k^2 \mid \PnBn^0, \PnBn^1, B_n\bi] \\
        & \leq (4M)^2 \Esp\bii[\bi(\psi^*_1(Z) - \psihat_k(\PnBn^0)(Z)\bi)^2 \mid \PnBn^0, \PnBn^1, B_n\bii] \\
        & = M_2 \Esp\bi[X_k \mid \PnBn^0, \PnBn^1, B_n \bi].
    \end{align*}
\end{proof}

\subsection{Proof of Theorem~\ref{thm::CV_censored_pobs}}
We write
    \begin{align*}
        0 & \leq \thetatildepo(\khat\po) - \theta^*_1 \\
        & = \EspBn \int \bii(L\po\bi(\psihat_{\khat\po}(\PnBn^0),\PnBn^1,o\bi) - L\po\bi(\psi^*_1,\PnBn^1,o\bi)\bii) dP(o) \\
        & \leq \EspBn \int \bii(L\po\bi(\psihat_{\khat\po}(\PnBn^0),\PnBn^1,o\bi) - L\po\bi(\psi^*_1,\PnBn^1,o\bi) \bii) dP(o) \\
        & \quad - (1 + \gamma) \EspBn \int \bii(L\po\bi(\psihat_{\khat\po}(\PnBn^0),\PnBn^1,o\bi) - L\po\bi(\psi^*_1,\PnBn^1,o\bi) \bii) d\PnBn^2(o) \\
        & \quad + (1 + \gamma) \EspBn \int \bii(L\po\bi(\psihat_{\ktilde\po}(\PnBn^0),\PnBn^1,o\bi) - L\po\bi(\psi^*_1,\PnBn^1,o\bi) \bii) d\PnBn^2(o) \\
        & \quad - (1 + 2\gamma) \EspBn \int \bii(L\po\bi(\psihat_{\ktilde\po}(\PnBn^0),\PnBn^1,o\bi) - L\po\bi(\psi^*_1,\PnBn^1,o\bi) \bii) dP(o) \\
        & \quad + (1 + 2\gamma) \EspBn \int \bii(L\po\bi(\psihat_{\ktilde\po}(\PnBn^0),\PnBn^1,o\bi) - L\po\bi(\psi^*_1,\PnBn^1,o\bi) \bii) dP(o)
    \end{align*}
where the first inequality follows by definition of the optimal risk $\theta^*_1$~\eqref{eq::optimal_risk_pobs} and the second by definition of $\khat\po$~\eqref{eq::cv_selector_pobs}, which implies $\thetahatpo(\khat\po) \leq \thetahatpo(\ktilde\po)$. Denote the first two terms in the last expression by $R_{\khat\po,n}$ and the third and fourth terms by $T_{\ktilde\po,n}$; the last term is the benchmark risk difference $\bi(1+2\gamma\bi)\bi(\thetatildepo(\ktilde\po) - \theta^*_1\bi)$. Hence,
\begin{equation}\label{eq::thm2_proof}
    0 \leq \thetatildepo(\khat\po) - \theta^*_1 \leq \bi(1+2\gamma\bi) \bi(\thetatildepo(\ktilde\po) - \theta^*_1\bi) + R_{\khat\po,n} + T_{\ktilde\po,n}.
\end{equation}
Let
\begin{align*}
    \Htilde{k} & = \int \bii(L\po\bi(\psihat_k(\PnBn^0),\PnBn^1,o\bi) - L\po\bi(\psi^*_1,\PnBn^1,o\bi)\bii) dP(o) \\
    \Hhat{k} & = \int \bii(L\po(\psihat_k(\PnBn^0),\PnBn^1,o) - L\po(\psi^*_1,\PnBn^1,o)\bii) d\PnBn^2(o)
\end{align*}
and
\begin{align*}
    R_{k,n}(B_n) & = \bi(1+\gamma\bi)\bi(\Htilde{k} - \Hhat{k} \bi) - \gamma \Htilde{k} \\
    T_{k,n}(B_n) & = \bi(1+\gamma\bi)\bi(\Hhat{k} - \Htilde{k} \bi) - \gamma \Htilde{k}.
\end{align*}
Note that $\Htilde{k} \geq 0$ by definition of $\psi^*_1$ as the risk minimizer. With these notations, we set $R_{k,n} = \EspBn[R_{k,n}(B_n)]$ and $T_{k,n} = \EspBn[T_{k,n}(B_n)]$. 

We have
\begin{align*}
    & \Prob\bi(R_{\khat\po,n}(B_n) > s \mid \PnBn^0, \PnBn^1, B_n\bi) \\
    & \quad = \Prob\biii(\Htilde{\khat\po} - \Hhat{\khat\po} > \frac{1}{1+\gamma}(s+\gamma \Htilde{\khat\po}) \midii \PnBn^0, \PnBn^1, B_n \biii) \\
    & \quad \leq K_n \max_{k \in \{1,\dots,K_n\}} \Prob\biii(\Htilde{k} - \Hhat{k} > \frac{1}{1+\gamma}(s+\gamma \Htilde{k}) \midii \PnBn^0, \PnBn^1, B_n \biii).
\end{align*}
Consider the random variables $X_k = L\po\bi(\psihat_k(\PnBn^0),\PnBn^1,O\bi) - L\po\bi(\psi^*_1,\PnBn^1,O\bi)$, where $O$ is independent from $\PnBn^0$, $\PnBn^1$, $B_n$. 
Let $X_{k,i} = L\po\bi(\psihat_k(\PnBn^0),\PnBn^1,O_i\bi) - L\po\bi(\psi^*_1,\PnBn^1,O_i\bi)$, $i=1,\dots,n_2$. Conditionally on $\PnBn^0, \PnBn^1, B_n$, the $X_{k,i}$'s are $n_2$ i.i.d. copies of $X_k$ in the validation set. Note that $\Hhat{k} = \sum_{i=1}^{n_2} X_{k,i} /n_2$ and $\Htilde{k} = \Esp\bi[X_k \mid \PnBn^0, \PnBn^1, B_n\bi]$, so that $\Htilde{k} - \Hhat{k} = \Esp\bi[X_k \mid \PnBn^0, \PnBn^1, B_n\bi] - \sum_{i=1}^{n_2} X_{k,i}/n_2$ represents an empirical mean of $n_2$ centered conditionally i.i.d. random variables. For the quadratic loss function, $\lvert X_k \rvert \leq 8M^2$ almost surely, from Assumption~\eqref{eq::assumption_thm_2}. From Lemma~\ref{lemma::quadratic_var_bound_pobs},
\[
    \sigma_k^2 := \Var\bi(X_k \mid \PnBn^0, \PnBn^1, B_n\bi) \leq M_2 \Esp\bi[X_k \mid \PnBn^0, \PnBn^1, B_n\bi] = M_2 \Htilde{k}.
\]
For $s > 0$, $M_1 = 8 M^2$, by Bernstein's inequality~\citep[see Lemma 1 in][]{dudoit_asymptotics_2005}, as the $X_{k,i}$'s are independent conditionally on $\PnBn^0, \PnBn^1, B_n$,
\begin{align*}
    \Prob\bi(R_{k,n}(B_n) > s \mid \PnBn^0, \PnBn^1, B_n\bi) & = \Prob\biii(\Htilde{k} - \Hhat{k} > \frac{1}{1+\gamma}(s+\gamma \Htilde{k}) \midii \PnBn^0, \PnBn^1, B_n \biii) \\
    & \leq \Prob\biii(\Htilde{k} - \Hhat{k} > \frac{1}{1+\gamma}(s+\gamma \sigma_k^2 / M_2) \midii \PnBn^0, \PnBn^1, B_n \biii) \\
    & \leq \exp\left(- \frac{n_2}{2(1+\gamma)^2} \frac{(s + \gamma \sigma_k^2 / M_2)^2}{\sigma_k^2 + \frac{M_1}{3(1+\gamma)} (s + \gamma \sigma_k^2 / M_2) } \right).
\end{align*}
Note that
\begin{align*}
    \frac{(s + \gamma \sigma_k^2 / M_2)^2}{\sigma_k^2 + \frac{M_1}{3(1+\gamma)} (s + \gamma \sigma_k^2 / M_2) } = \frac{s + \gamma \sigma_k^2 / M_2}{\frac{\sigma_k^2}{s + \gamma \sigma_k^2 / M_2} + \frac{M_1}{3(1+\gamma)}  } \geq \frac{s + \gamma \sigma_k^2 / M_2}{\frac{M_2}{\gamma} + \frac{M_1}{3(1+\gamma)}} \geq \frac{s}{\frac{M_2}{\gamma} + \frac{M_1}{3(1+\gamma)}}\cdot
\end{align*}
Then for $s > 0$ and $c(M,\gamma) = 2(1+ \gamma)^2 \left(\frac{M_1}{3} + \frac{M_2}{\gamma} \right)$,
\begin{equation*}
    \Prob\bi(R_{\khat\po,n}(B_n) > s \mid \PnBn^0, \PnBn^1, B_n\bi) \leq K_n \exp\left( - \frac{n_2}{c(M,\gamma)}s\right).
\end{equation*}
The same bound applies to the marginal probabilities $\Prob(R_{\khat\po,n}(B_n) > s)$. For all $u \geq 0$,
\[
    \Esp\bi[R_{\khat\po,n}\bi] 
    \leq u + \int_u^{\infty} K_n \exp\left( - \frac{n_2}{c(M,\gamma)}s\right) ds.
\]
This function of $u$ achieves a minimum of $c(M,\gamma)(1+\log(K_n))/n_2$ at $u = c(M,\gamma) \log(K_n)/n_2$. Thus,
\[
    \Esp\bi[R_{\khat\po,n}\bi] \leq c(M,\gamma) \frac{1 + \log(K_n)}{n_2}.
\]
The same bound applies to $\Esp\bi[T_{\ktilde\po,n}\bi]$. Taking the expectation in Equation~\eqref{eq::thm2_proof} gives the first asymptotic result stated in Equation~\eqref{eq::thm2_asym1}. This result, combined with Lemma 2 in \citet{dudoit_asymptotics_2005}, allows to derive the second asymptotic result given in Equation~\eqref{eq::thm2_asym2}.

\subsection{Proof of Corollary~\ref{cor::CV_censored_pobs}}


In the calculations below, the expectations are written with respect to the variables $O^* = (T^*,Z)$ and $O = (T^* \wedge C,\indicator{T^* \leq C},Z)$, for a triplet $(T^*,C,Z)$ independent from the data distribution $B_n$ and from the dataset $P_{B_n} = \{\PnBn^0,\PnBn^1,\PnBn^2\}$.
The results essentially rely on Equation~\eqref{eq::pobs_split_cond_mean}. On the one hand,
    \begin{align*}
        & \thetatildepo(\khat\po) - \thetatilde(\khat\po) \\
        & \quad = \EspBn \bii[\Esp\bii[ L\po\bi(\psihat_{\khat\po}(\PnBn^0),\PnBn^1,O\bi) \midi P_{B_n},B_n \bii]\bii] - \EspBn \bii[\Esp\bii[ L\bi(\psihat_{\khat\po}(\PnBn^0),O^*\bi) \midi P_{B_n},B_n \bii]\bii] \\
        & \quad = \EspBn\bii[\Esp\bii[\bi(\Gamma_O(\PnBn^1) - \psihat_{\khat\po}(\PnBn^0)(Z)\bi)^2 - \bi(T^* \wedge \tau - \psihat_{\khat\po}(\PnBn^0)(Z)\bi)^2 \midi P_{B_n},B_n \bii]\bii] \\
        & \quad = \EspBn\bii[\Esp\bii[\Gamma_O(\PnBn^1)^2 - (T^* \wedge \tau)^2 - 2 \psihat_{\khat\po}(\PnBn^0)(Z) \bi(\Gamma_O(\PnBn^1) - T^* \wedge \tau\bi) \midi P_{B_n},B_n \bii]\bii] \\
        & \quad = \EspBn\bii[\Esp\bii[\Esp\bi[\Gamma_O(\PnBn^1)^2 \mid Z, \PnBn^1, B_n\bi] - \Esp\bi[(T^* \wedge \tau)^2 \mid Z\bi] \\
        & \qquad - 2 \psihat_{\khat\po}(\PnBn^0)(Z) \bi(\Esp\bi[\Gamma_O(\PnBn^1) \mid Z, \PnBn^1,B_n \bi] - \Esp\bi[T^* \wedge \tau \mid Z \bi]\bi) \midi P_{B_n},B_n \bii]\bii] \\
        & \quad = \EspBn\bii[\Esp\bii[\Esp\bi[\Gamma_O(\PnBn^1)^2 \mid Z, \PnBn^1, B_n\bi] - \Esp\bi[(T^* \wedge \tau)^2 \mid Z\bi] \\
        & \qquad - 2 \psihat_{\khat\po}(\PnBn^0)(Z) \Esp[\xin \mid Z, \PnBn^1, B_n] \midi P_{B_n},B_n \bii]\bii].
    \end{align*}
    Similarly,
    \begin{align*}
        \thetatildepo(\ktilde\po) - \thetatilde(\ktilde\po) & = \EspBn\bii[\Esp\bii[\Esp\bi[\Gamma_O(\PnBn^1)^2 \mid Z, \PnBn^1, B_n\bi] - \Esp\bi[(T^* \wedge \tau)^2 \mid Z\bi] \\
        & \quad - 2 \psihat_{\ktilde\po}(\PnBn^0)(Z) \Esp[\xin \mid Z, \PnBn^1, B_n] \midi P_{B_n},B_n \bii]\bii].
    \end{align*}
    On the other hand,
    \begin{align*}
        & \theta^*_1 - \theta^* \\
        & \quad = \EspBn \bii[\Esp\bii[ L\po\bi(\psi^*_1,\PnBn^1,O\bi) \midi P_{B_n},B_n \bii]\bii] - \Esp\bii[ L\bi(\psi^*,O^*\bi) \bii] \\
        & \quad = \EspBn\bii[\Esp\bii[\bi(\Gamma_O(\PnBn^1) - \psi^*_1(Z)\bi)^2 - \bi(T^* \wedge \tau - \psi^*(Z)\bi)^2 \midi P_{B_n},B_n \bii]\bii] \\
        & \quad = \EspBn\bii[\Esp\bii[\Gamma_O(\PnBn^1)^2 - (T^* \wedge \tau)^2 + \psi^*_1(Z)^2 - \psi^*(Z)^2 \\
        & \qquad - 2 \Gamma_O(\PnBn^1) \psi^*_1(Z) + 2 (T^* \wedge \tau) \psi^*(Z) \midi P_{B_n},B_n \bii]\bii] \\
        & \quad = \EspBn\bii[\Esp\bii[ \Esp\bi[\Gamma_O(\PnBn^1)^2 \mid Z, \PnBn^1, B_n \bi] - \Esp\bi[(T^* \wedge \tau)^2 \mid Z \bi] - \bi(\psi^*_1(Z)^2 - \psi^*(Z)^2 \bi) \midi P_{B_n},B_n \bii]\bii] \\
        & \quad = \EspBn\bii[\Esp\bii[ \Esp\bi[\Gamma_O(\PnBn^1)^2 \mid Z, \PnBn^1, B_n \bi] - \Esp\bi[(T^* \wedge \tau)^2 \mid Z \bi] \\
        & \qquad - \Esp\bi[\xin \mid Z , \PnBn^1,B_n\bi]^2 - 2 \psi^*(Z) \Esp\bi[\xin \mid Z , \PnBn^1,B_n\bi] \midi P_{B_n},B_n \bii]\bii].
    \end{align*}
    Thus, if we let
    \begin{align*}
        \phi_{n_1}(k) & = \EspBn\bii[\Esp\bii[\Esp\bi[\xin \mid Z , \PnBn^1,B_n\bi]^2 \\
        & \qquad + 2\bi(\psi^*(Z) - \psihat_{k}(\PnBn^0)(Z)\bi)\Esp\bi[\xin \mid Z , \PnBn^1,B_n\bi] \midi P_{B_n},B_n \bii]\bii],
    \end{align*}
    we have, from Equation~\eqref{eq::fin_sam_pobs}, 
    \begin{align*}
        0 & \leq \Esp\bi[\thetatildepo(\khat\po) - \theta^*_1\bi] \leq (1+2 \gamma) \Esp\bi[ \thetatildepo(\ktilde\po) - \theta^*_1\bi] + 2 c(M,\gamma) \frac{1 + \log(K_n)}{n_2} \\
        \iff 0 & \leq \Esp\bi[\thetatilde(\khat\po) - \theta^* + \phi_{n_1}(\khat\po)\bi] \\
        & \leq (1+2 \gamma) \Esp\bi[ \thetatilde(\ktilde\po) - \theta^* + \phi_{n_1}(\ktilde\po)\bi] + 2 c(M,\gamma) \frac{1 + \log(K_n)}{n_2} \\
        \iff 0 & \leq \Esp\bi[\thetatilde(\khat\po) - \theta^*\bi] \\
        & \leq (1+2 \gamma) \Esp\bi[ \thetatilde(\ktilde\po) - \theta^*\bi] + \Esp\bi[(1+2\gamma) \phi_{n_1}(\ktilde\po) - \phi_{n_1}(\khat\po)\bi] + 2 c(M,\gamma) \frac{1 + \log(K_n)}{n_2}.
    \end{align*}
    
    To obtain the asymptotic results, we rely on Lemma~\ref{lemma::argmin_lim} below. It implies that $\thetatilde(\ktilde\po) \to \thetatilde(\ktilde)$ in distribution as $n \to \infty$. In addition, we assumed the pseudo-observations and candidate learners to be bounded (Assumption~\eqref{eq::assumption_thm_2}). Thus $\Esp[\thetatilde(\ktilde\po) - \thetatilde(\ktilde)] \to 0$ as $n \to \infty$. Besides, $\phi_{n_1}(k) = o_\Prob(1)$ and is bounded almost surely for all $k$, so that we also have $\Esp\bi[(1+2\gamma) \phi_{n_1}(\ktilde\po) - \phi_{n_1}(\khat\po)\bi] \to 0$ as $n \to \infty$.
    Hence the first asymptotic result (Equation~\eqref{eq::cor_asym1}) which, combined with Lemma 2 in \citet{dudoit_asymptotics_2005}, allows to derive the second asymptotic result (Equation~\eqref{eq::cor_asym2}).
    

\begin{lemma}\label{lemma::argmin_lim}
    $\thetatilde(\ktilde\po) \to \thetatilde(\ktilde)$ in probability as $n \to \infty$.
\end{lemma}

\begin{proof}
    Let $f_k(n) = \Esp_{B_n}\bii[\Esp\bii[\bi(T^* \wedge \tau - \psihat_k(\PnBn^0)(Z)\bi)^2 \midi \PnBn^0, B_n \bii]\bii]$. We have $f_k(n) = \thetatilde(k)$, so that
    \begin{align*}
        \ktilde & = \argmin_{k \in \{1,\dots,K_n\}} f_k(n).
    \end{align*}
    On the other hand, using arguments similar to the proof of Lemma~\ref{lemma::risk_minimizer_quad_pobs}, we find
    \begin{align*}
        \ktilde\po & = \argmin_{k \in \{1,\dots,K_n\}} \thetatildepo(k) \\
        & = \argmin_{k \in \{1,\dots,K_n\}} \Esp_{B_n}\bii[\Esp\bii[\bi(\Gamma_O(\PnBn^1) - \psihat_k(\PnBn^0)(Z)\bi)^2 \midi \PnBn^0, \PnBn^1, B_n \bii]\bii] \\
        & = \argmin_{k \in \{1,\dots,K_n\}} \Esp_{B_n}\biii[\Esp\biii[\bii(\Esp\bi[\Gamma_O(\PnBn^1) \mid Z, \PnBn^1, B_n\bi] - \psihat_k(\PnBn^0)(Z)\bii)^2 \midii \PnBn^0, \PnBn^1, B_n \biii]\biii].
    \end{align*}
    Using Equation~\eqref{eq::pobs_split_cond_mean}, we write
    \begin{align*}
        & \Esp\biii[\bii(\Esp\bi[\Gamma_O(\PnBn^1) \mid Z, \PnBn^1, B_n\bi] - \psihat_k(\PnBn^0)(Z)\bii)^2 \midii \PnBn^0, \PnBn^1, B_n \biii] \\
        & = \Esp\biii[\bii(\Esp\bi[T^* \wedge \tau \mid Z\bi] + \Esp\bi[\xin \mid Z, \PnBn^1, B_n \bi] - \psihat_k(\PnBn^0)(Z)\bii)^2 \midii \PnBn^0, \PnBn^1, B_n \biii] \\
        & = \Esp\biii[\bii(\Esp\bi[T^* \wedge \tau \mid Z \bi] - T^* \wedge \tau \bii)^2 \midii \PnBn^0, \PnBn^1, B_n \biii] \\
        & \quad + \Esp\biii[\bii( T^* \wedge \tau + \Esp\bi[\xin \mid Z, \PnBn^1, B_n\bi] - \psihat_k(\PnBn^0)(Z) \bii)^2 \midii \PnBn^0, \PnBn^1, B_n \biii] \\
        & \quad + 2 \Esp\biii[\bii(\Esp\bi[T^* \wedge \tau \mid Z\bi] - T^* \wedge \tau\bii)\bii(T^* \wedge \tau + \Esp\bi[\xin \mid Z, \PnBn^1, B_n\bi] - \psihat_k(\PnBn^0)(Z)\bii)\midii \PnBn^0, \PnBn^1, B_n \biii].
    \end{align*}
    The second term can be developed into
    \begin{align*}
        & \Esp\biii[\bii( T^* \wedge \tau + \Esp\bi[\xin \mid Z, \PnBn^1, B_n \bi] - \psihat_k(\PnBn^0)(Z)\bii)^2 \midii \PnBn^0, \PnBn^1, B_n \biii] \\
        & = \Esp\biii[\bii( T^* \wedge \tau - \psihat_k(\PnBn^0)(Z) \bii)^2 \midii \PnBn^0, B_n \biii] + \Esp\biii[\Esp\bi[\xin \mid Z, \PnBn^1, B_n\bi]^2 \midii \PnBn^1, B_n \biii] \\
        & \quad + 2 \Esp\biii[\bi(T^* \wedge \tau\bi) \Esp\bi[\xin \mid Z, \PnBn^1, B_n\bi] \midii \PnBn^1, B_n\biii] - 2 \Esp\biii[\psihat_k(\PnBn^0)(Z) \xin \midii \PnBn^0, \PnBn^1, B_n \biii],
    \end{align*}
    while the third term simplifies into
    \begin{align*}
        & \Esp\biii[\bii(\Esp\bi[T^* \wedge \tau \mid Z\bi] - T^* \wedge \tau\bii)\bii(T^* \wedge \tau + \Esp\bi[\xin \mid Z, \PnBn^1, B_n\bi] - \psihat_k(\PnBn^0)(Z)\bii)\midii \PnBn^0, \PnBn^1, B_n \biii] \\
        & = \Esp\biii[\bii(\Esp\bi[T^* \wedge \tau \mid Z\bi] - T^* \wedge \tau \bii)\bii(T^* \wedge \tau\bii)\biii] \\
        & \quad + \Esp\biii[\bii(\Esp\bi[T^* \wedge \tau \mid Z\bi] - T^* \wedge \tau \bii)\bii(\Esp\bi[\xin \mid Z, \PnBn^1, B_n\bi] - \psihat_k(\PnBn^0)(Z)\bii)\midii \PnBn^0, \PnBn^1, B_n \biii] \\
        & = \Esp\biii[\bii(\Esp\bi[T^* \wedge \tau \mid Z\bi] - T^* \wedge \tau \bii)\bii(T^* \wedge \tau\bii)\biii].
    \end{align*}
    Removing all terms independent from $k$, we find
    \begin{align*}
        \ktilde\po & = \argmin_{k \in \{1,\dots,K_n\}} \bii\{ f_k(n) + \varepsilon_k(n) \bii\}
    \end{align*}
    where $\varepsilon_k(n) = - 2 \Esp_{B_n}\bii[\Esp\bii[\psihat_k(\PnBn^0)(Z)\xin \midi \PnBn^0, \PnBn^1, B_n \bii]\bii] = o_\Prob(1)$ as the $\psihat_k(\PnBn^0)$'s are bounded by assumption and $\xin = o_\Prob(1)$.

    Let $f_{(1)_n}(n)<\dots<f_{(K_n)_n}(n)$ denote the order statistics of  $f_1(n),\dots,f_{K_n}(n)$. The order statistics depends on $n$ both by the dependence of $f_1,\dots,f_{K_n}$ on $n$ and by the number $K_n$ increasing with $n$. 
    Asymptotically, either the $\{f_k(n) + \varepsilon_k(n)\}$'s are ordered in the same way as the $f_k(n)$'s, i.e. $\ktilde\po = \ktilde$, or the inversions are negligible in terms of risk. 
    Let $C > 0$.
    First suppose that there exist $k$, $k'$ and $n$ such that $\lvert f_k(n) - f_{k'}(n) \rvert < C$. This is equivalent to 
    \[
        \lvert \thetatilde(k) - \thetatilde(k') \rvert < C.
    \]
    
    Now suppose that for all $k$, $k'$ and for all $n$, $\lvert f_k(n) - f_{k'}(n) \rvert > C$. For all $k = 1,\dots,K_n$, $\varepsilon_k(n) =  o_\Prob(1)$, thus there exists $N > 0$ such that for all $n \geq N$, for all $k = 1,\dots,K_n$, $\vert \varepsilon_k(n) \rvert < C/2$ with probability one. Thus, for all $n \geq N$, if $f_{(k)_n}(n) < f_{(k')_n}(n)$, then
    \begin{align*}
        f_{(k)_n}(n) + \varepsilon_{(k)_n}(n) & < f_{(k)_n}(n) + \frac{C}{2} < \frac{f_{(k')_n}(n) + f_{(k)_n}(n)}{2} \\
        f_{(k')_n}(n) + \varepsilon_{(k')_n}(n) & > f_{(k)_n}(n) - \frac{C}{2} > \frac{f_{(k')_n}(n) + f_{(k)_n}(n)}{2},
    \end{align*}
    i.e. $f_{(k)_n}(n) + \varepsilon_{(k)_n}(n) < f_{(k')_n}(n) + \varepsilon_{(k')_n}(n)$. The order is preserved asymptotically, i.e. $\ktilde\po = \ktilde$, with probability one.
    We conclude by making $C$ tend towards $0$.

\end{proof}


\bibliography{biblio}

\end{document}


\maketitle

\section{Corollary of Theorem 2 in the form of a ratio}

Asymptotic results in Corollary~2.1 of the main document are given for the difference between the terms $\thetatilde(\khat\po) - \theta^*$ and $\thetatilde(\ktilde) - \theta^*$. To get asymptotic results in terms of ratio, additional assumptions are needed, see Corollary~\ref{cor::CV_censored_pobs2} below. These new assumptions essentially require that the risk of the best candidate learner does not converge to the optimal risk at a higher speed than the speed of convergence of the pseudo-observations. If the risk of the best candidate learner does not converge to the optimal risk (i.e. we do not have $\thetatilde(\ktilde) \to \theta^*$), which is often the case in real cases scenarios, then the assumptions~\eqref{eq::cor2_hyp1}, \eqref{eq::cor2_hyp2}, \eqref{eq::cor2_hyp3} and \eqref{eq::cor2_hyp4} are automatically fulfilled and we only need $\log(K_n)/(n p_{2,n}) \to 0$ as $n \to \infty$ to get the asymptotic results.
\setcounter{theorem}{2}
\setcounter{corollary}{1}
\begin{corollary} \label{cor::CV_censored_pobs2}
    Same setup and assumptions as in Theorem~2. 
    
    \noindent If 
    \[
        \frac{\log(K_n)}{n p_{2,n} \Esp[ \thetatilde(\ktilde) - \theta^*]} \to 0 \text{ as } n \to \infty,
    \]
    and
    \begin{equation}\label{eq::cor2_hyp1}
        \frac{\Esp\bi[\thetatilde(\ktilde\po) - \thetatilde(\ktilde)\bi]}{\Esp[ \thetatilde(\ktilde) - \theta^*]} \to 0 \text{ as } n \to \infty,
    \end{equation}
    and
    \begin{equation}\label{eq::cor2_hyp2}
        \frac{\Esp\bi[(1+2\gamma) \phi_{n_1}(\ktilde\po) - \phi_{n_1}(\khat\po)\bi]}{\Esp[ \thetatilde(\ktilde) - \theta^*]} \to 0 \text{ as } n \to \infty,
    \end{equation}
    then
    \begin{equation*}
        \frac{\Esp[ \thetatilde(\khat\po) - \theta^*]}{\Esp[ \thetatilde(\ktilde) - \theta^*]} \to 1 \text{ as } n \to \infty.
    \end{equation*}
    Similarly, if
    \[
        \frac{\log(K_n)}{n p_{2,n} (\thetatilde(\ktilde) - \theta^*)} \to 0 \text{ in probability as } n \to \infty,
    \]
    and
    \begin{equation}\label{eq::cor2_hyp3}
        \frac{\thetatilde(\ktilde\po) - \thetatilde(\ktilde)}{ \thetatilde(\ktilde) - \theta^*} \to 0 \text{ in probability as } n \to \infty,
    \end{equation}
    and
    \begin{equation}\label{eq::cor2_hyp4}
        \frac{(1+2\gamma) \phi_{n_1}(\ktilde\po) - \phi_{n_1}(\khat\po)}{\thetatilde(\ktilde) - \theta^*} \to 0 \text{ in probability as } n \to \infty,
    \end{equation}
    then
    \begin{equation*}
        \frac{\thetatilde(\khat\po) - \theta^*}{\thetatilde(\ktilde) - \theta^*} \to 1 \text{ in probability as } n \to \infty.
    \end{equation*}
\end{corollary} 

\section{Distribution of the weights in the Super Learner (complement to Section~4.2)}

In Section~4.2 of the main document, we give numerical results for the implementation of the super learner with standard or split pseudo-observations.  The training algorithms considered for the two methods are: the Linear Model (LM), the Lasso, the Generalized Additive Model (GAM), the Random Forests (RF) and the Neural Network (NN).
The performances of the two methods are further compared with the oracle method that applies the same algorithms directly to the unoberved true event times. 
We recall that two simulation scenarios are considered, where data are simulated according to a Cox model, with higher complexity in the second scheme (including more variables and interaction terms).
All super learners are trained using $6$-fold cross-validation.
The results in terms of mean squared error (MSE) are displayed in Figure~4 in the main document, for the three different methods in the two simulation schemes with increasing sample sizes $n\in\{100,200,300,400,500\}$. The whole procedure is repeated $80$ times for each sample size. Exact values can be found in Table~\ref{tab::mse_graph}. 
We complement these results with the corresponding distribution of the weights assigned to each learner, displayed in Figure~\ref{fig::beta_graph}. The median and the first and third quantiles of the assigned weights are reported in the figure.

\section{Applications on real data: additional results}

In Section 5 of the main document, we apply our prediction algorithm to maintenance and colon cancer datasets and use methods from a previous work~\citep[see][]{cwiling_comprehensive_2023} in order to analyze the results. These methods rely on the estimation of the censoring survival function, thus we first investigate the possible dependence between censoring and covariates. To that aim we fit the Cox model and the Random Survival Forests (RSF) to the data using the censoring time as the outcome. P-values are reported for the Cox model and measures of variables importance per permutation are reported for the RSF in Table~\ref{tab::cens_mtnc} for the \texttt{maintenance} dataset and in Table~\ref{tab::cens_colon} for the \texttt{colon} dataset. The censoring survival function is then estimated for the implementation of three IPCW methods used to estimate the MSE of the methods, construct predictions intervals, and evaluate the importance of the variables. In particular, a prediction interval for the true restricted time to event of level $90\%$ is computed for each data point. These intervals are computed with the IPCW Rank-One-Out Split Conformal algorithm. An illustration of these intervals, for some items randomly chosen in the dataset, is provided in Figure~\ref{fig::pred_intvs_mtnc} for the \texttt{maintenance} dataset, and in Figure~\ref{fig::pred_intvs_colon} for the \texttt{colon} dataset.

\bibliography{biblio}

\begin{table}[!ht]
    \begin{subtable}[h]{\textwidth}
        \centering
        \begin{tabular}{|r|lllll|}
        	\hline
        	n & 100 & 200 & 300 & 400 & 500 \\ 
        	\hline
        	LM & 0.197 (0.007) & 0.189 (0.006) & 0.191 (0.006) & 0.187 (0.005) & 0.188 (0.005) \\ 
        	Lasso & 0.197 (0.007) & 0.189 (0.005) & 0.191 (0.006) & 0.187 (0.005) & 0.188 (0.005) \\ 
        	GAM & 0.190 (0.009) & 0.182 (0.005) & 0.183 (0.006) & 0.179 (0.004) & 0.179 (0.005) \\ 
        	RF & 0.239 (0.017) & 0.169 (0.009) & 0.144 (0.008) & 0.128 (0.007) & 0.118 (0.008) \\ 
        	NN & \textbf{0.097} (0.009) & \textbf{0.084} (0.006) & \textbf{0.085} (0.005) & \textbf{0.083} (0.006) & \textbf{0.081} (0.007) \\ 
        	\hline
        	SL & \textbf{0.097} (0.008) & \textbf{0.084} (0.006) & \textbf{0.085} (0.004) & \textbf{0.083} (0.006) & \textbf{0.081} (0.006) \\ 
        	\hline
        \end{tabular}
        \caption{Simulation scheme 1, True restricted event times}
        \label{tab::scheme_B_oracle}
    \end{subtable}
    
    \begin{subtable}[h]{\textwidth}
        \centering
        \begin{tabular}{|r|lllll|}
        	\hline
        	n & 100 & 200 & 300 & 400 & 500 \\ 
        	\hline
        	LM & \textbf{0.212} (0.011) & 0.199 (0.009) & 0.196 (0.006) & 0.191 (0.006) & 0.190 (0.005) \\ 
        	Lasso & 0.213 (0.012) & 0.199 (0.009) & 0.196 (0.006) & 0.191 (0.005) & 0.190 (0.005) \\ 
        	GAM & 0.213 (0.010) & 0.199 (0.010) & 0.191 (0.007) & 0.185 (0.006) & 0.185 (0.006) \\ 
        	RF & 0.297 (0.029) & 0.229 (0.016) & 0.202 (0.016) & 0.184 (0.012) & 0.174 (0.011) \\
        	NN & 0.213 (0.062) & \textbf{0.140} (0.021) & \textbf{0.115} (0.012) & \textbf{0.105} (0.009) & \textbf{0.098} (0.008) \\ 
        	\hline
        	SL & \textbf{0.170} (0.020) & \textbf{0.133} (0.015) & \textbf{0.115} (0.012) & \textbf{0.104} (0.007) & 0.099 (0.007) \\ 
        	SL (split) & \textbf{0.173} (0.025) & \textbf{0.133} (0.014) & 0.118 (0.013) & \textbf{0.105} (0.008) & \textbf{0.098} (0.007) \\ 
        	\hline
        \end{tabular}
        \caption{Simulation scheme 1, Pseudo-observations}
        \label{tab::scheme_B_pobs}
     \end{subtable}

     \begin{subtable}[h]{\textwidth}
        \centering
        \begin{tabular}{|r|lllll|}
        	\hline
        	n & 100 & 200 & 300 & 400 & 500 \\ 
        	\hline  
        	LM & 0.229 (0.012) & 0.211 (0.007) & \textbf{0.206} (0.007) & 0.200 (0.006) & 0.200 (0.006) \\ 
        	Lasso & \textbf{0.224} (0.012) & \textbf{0.208} (0.007) & \textbf{0.206} (0.007) & 0.200 (0.006) & 0.198 (0.006) \\
        	GAM & 0.239 (0.014) & 0.216 (0.007) & 0.210 (0.007) & 0.203 (0.006) & 0.202 (0.006) \\ 
        	RF & 0.257 (0.012) & 0.221 (0.006) & 0.208 (0.007) & \textbf{0.197} (0.006) & \textbf{0.188 }(0.006) \\
        	NN & 0.424 (0.060) & 0.267 (0.030) & 0.218 (0.022) & 0.201 (0.015) & 0.195 (0.013) \\
        	\hline
        	SL & 0.225 (0.013) & \textbf{0.202} (0.009) & \textbf{0.194} (0.008) & \textbf{0.187} (0.007) & \textbf{0.180} (0.007) \\ 
        	\hline
        \end{tabular}
        \caption{Simulation scheme 2, True restricted event times}
        \label{tab::scheme_C_oracle}
    \end{subtable}
    
    \begin{subtable}[h]{\textwidth}
        \centering
        \begin{tabular}{|r|lllll|}
        	\hline
        	n & 100 & 200 & 300 & 400 & 500 \\ 
        	\hline
        	LM & 0.271 (0.029) & 0.229 (0.012) & 0.218 (0.010) & \textbf{0.209} (0.008) & 0.207 (0.007) \\ 
        	Lasso & \textbf{0.250} (0.023) & \textbf{0.220} (0.008) & \textbf{0.214} (0.008) & \textbf{0.209} (0.007) & 0.206 (0.007) \\ 
        	GAM & 0.302 (0.032) & 0.240 (0.012) & 0.223 (0.010) & 0.215 (0.007) & 0.212 (0.008) \\ 
        	RF & 0.284 (0.024) & 0.240 (0.011) & 0.221 (0.008) & 0.213 (0.007) & \textbf{0.203} (0.007) \\ 
        	NN & 0.530 (0.104) & 0.381 (0.048) & 0.302 (0.051) & 0.257 (0.044) & 0.238 (0.026) \\ 
        	\hline
        	SL & 0.257 (0.024) & \textbf{0.217} (0.007) & \textbf{0.209} (0.008) & \textbf{0.199} (0.008) & \textbf{0.193} (0.008) \\ 
        	SL (split) & 0.256 (0.026) & \textbf{0.217 }(0.009) & \textbf{0.209} (0.009) & \textbf{0.200} (0.007) & \textbf{0.194} (0.007) \\ 
        	\hline
        \end{tabular}
        \caption{Simulation scheme 2, Pseudo-observations}
        \label{tab::scheme_C_pobs}
     \end{subtable}
     \caption{Mean Squared Errors (MSE) for the restricted time prediction with algorithms implemented on the true uncensored observations (Tables \ref{tab::scheme_B_oracle} and \ref{tab::scheme_C_oracle}) and on the pseudo-observations (Tables \ref{tab::scheme_B_pobs} and \ref{tab::scheme_C_pobs}). The considered algorithms are: the linear model (LM), the Lasso, the generalized additive model (GAM), the random forests (RF), the neural network (NN) and the super learner (SL) combining all these algorithms. On the pseudo-observations, standard and split pseudo-observations-based super learners are implemented. The sample size of the training set ranges from $n=100$ to $n=500$ and two simulation schemes are considered. The MSEs of all the algorithms are computed on an independent test set of size $1000$ and the whole procedure is repeated $80$ times. The median value is reported, followed by the median absolute deviation in brackets. For each value of $n$, the smallest median MSE value among the candidate learners is highlighted in bold. The MSE values for the standard and split pseudo-observations-based super learners are highlighted in bold if they are lower than the best candidate learner.}
     \label{tab::mse_graph}
\end{table}

\begin{figure}[ht!]
    \centering
    
    \subfloat[Simulation scheme 1]{%
        \includegraphics[width=.91\linewidth]{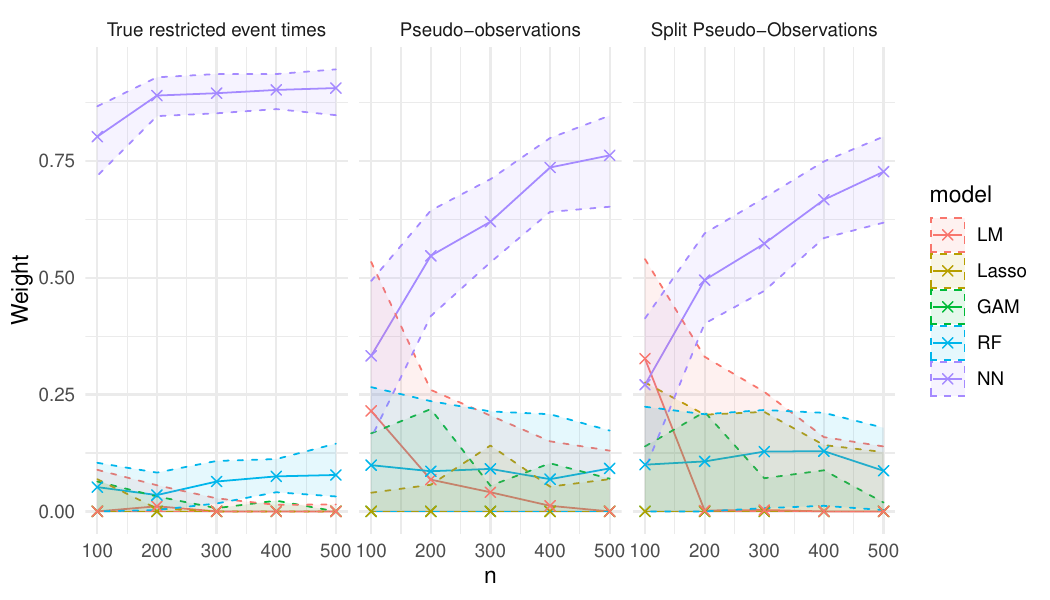}%
        \label{fig::beta_graph_B}%
    }
    
    \subfloat[Simulation scheme 2]{%
        \includegraphics[width=.91\linewidth]{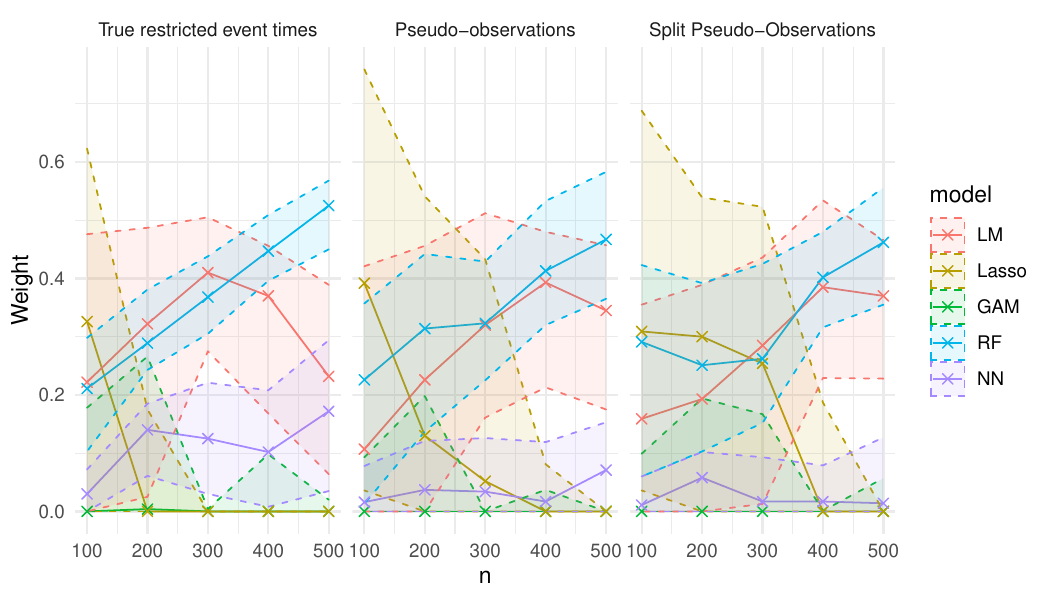}%
        \label{fig::beta_graph_C}%
    }

    \caption{Weights assigned to the candidate learners in the super learner implemented on the true uncensored observations (first column), the standard pseudo-observations (second column) and the split pseudo-observations (third column). The considered algorithms are: the linear model (LM), the Lasso, the generalized additive model (GAM), the random forests (RF) and the neural network (NN). The sample size of the training set ranges from $n=100$ to $n=500$ and two simulation schemes are considered in the top and bottom rows. The whole procedure is repeated $80$ times. Values of the median (cross), first and third quantiles (dashed lines) of the weights are reported.}
    \label{fig::beta_graph}
\end{figure}

\clearpage

\begin{table}[!ht]
    \begin{subtable}[h]{0.45\textwidth}
        \centering
        \begin{tabular}{|ll|r|}
            \hline
            Variable  &  & P-value \\
            \hline
            Pressure  &  & 0.568 \\
            Moisture  &  & 0.742 \\
            Temperature  &  & 0.705 \\
            Team  & B & 0.234 \\
              & C & 0.948 \\
            Provider  & 2 & 0.161 \\
              & 3 & 0.099 \\
              & 4 & 0.143 \\
            \hline
        \end{tabular}
        \caption{P-values for the Cox model}
        \label{fig::cens_cox_mtnc}
    \end{subtable}
    \hfill
    \begin{subtable}[h]{0.45\textwidth}
        \centering
        \begin{tabular}{|l|r|}
            \hline
            Variable  & Importance \\
            \hline
            Pressure  & -0.003 \\
            Moisture  & -0.009 \\
            Temperature  & -0.002 \\
            Team  & 0.004 \\
            Provider  & -0.003 \\
            \hline
        \end{tabular}
        \caption{Variable importance for the RSF}
        \label{fig::cens_rsf_mtnc}
     \end{subtable}
     \caption{Measures of dependence between censoring and covariates on the \texttt{maintenance} dataset. A Cox model and a RSF are fitted to the data using the censoring time as the outcome. P-values are reported for the Cox model while measures of variables importance per permutation are reported for the RSF. The global p-value for the Cox model is equal to $0.604$.}
     \label{tab::cens_mtnc}
\end{table}

\begin{figure}[!ht]
    \centering
    \includegraphics[scale=0.77]{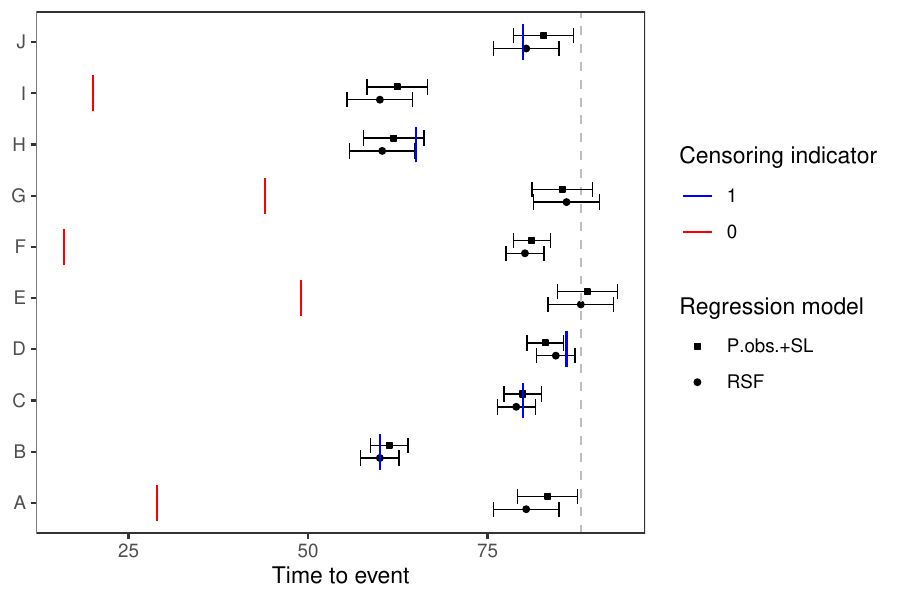}
    \caption{Prediction intervals at the $90\%$ level constructed with the IPCW Rank-One-Out Split Conformal algorithm~\citep[see][]{cwiling_comprehensive_2023} for the RSF and the super learner based on pseudo-observations on the \texttt{maintenance} dataset. All data are used in the training, and the prediction intervals are displayed for $10$ random machines. The grey dotted line represents the time horizon $\tau = 88$. The segments, red for censored and blue for uncensored times, are placed at the minimum between the observed times and $\tau$.}
    \label{fig::pred_intvs_mtnc}
\end{figure}

\begin{table}[!ht]
    \begin{subtable}[h]{0.45\textwidth}
        \centering
        \begin{tabular}{|ll|r|}
            \hline
            Variable  &  & P-value \\
            \hline
            Treatment  & Levamisole & 0.187 \\
             & Levamisole+5-FU & 0.144 \\
            Sex  &  & 0.238 \\
            Age  &  & 0.437 \\
            Obstruction  &  & 0.580 \\
            Perforation  &  & 0.391 \\
            Adherence  &  & 0.576 \\
            Nodes  &  & 0.117 \\
            Differentiation  & 2 & 0.882 \\
              & 3 & 0.741 \\
            Spread  & 2 & 0.147 \\
              & 3 & 0.085 \\
              & 4 & 0.760 \\
            Surgery & & 0.108 \\
            \hline
        \end{tabular}
        \caption{P-values for the Cox model}
        \label{fig::cens_cox_colon}
    \end{subtable}
    \hfill
    \begin{subtable}[h]{0.45\textwidth}
        \centering
        \begin{tabular}{|l|r|}
            \hline
            Variable  & Importance \\
            \hline
            Treatment  & -0.003 \\
            Sex  & -0.002 \\
            Age  & 0.004 \\
            Obstruction  & -0.005 \\
            Perforation  & -0.001 \\
            Adherence  & -0.002 \\
            Nodes  & 0.005 \\
            Differentiation  & -0.011 \\
            Spread  & 0.009 \\
            Surgery & 0.015 \\
            \hline
        \end{tabular}
        \caption{Variable importance for the RSF}
        \label{fig::cens_rsf_colon}
     \end{subtable}
     \caption{Measures of dependence between censoring and covariates on the \texttt{colon} dataset. A Cox model and a RSF are fitted to the data using the censoring time as the outcome. P-values are reported for the Cox model while measures of variables importance per permutation are reported for the RSF. The global p-value for the Cox model is equal to $0.383$.}
    \label{tab::cens_colon}
\end{table}


\begin{figure}[!ht]
    \centering
    \includegraphics[scale=0.77]{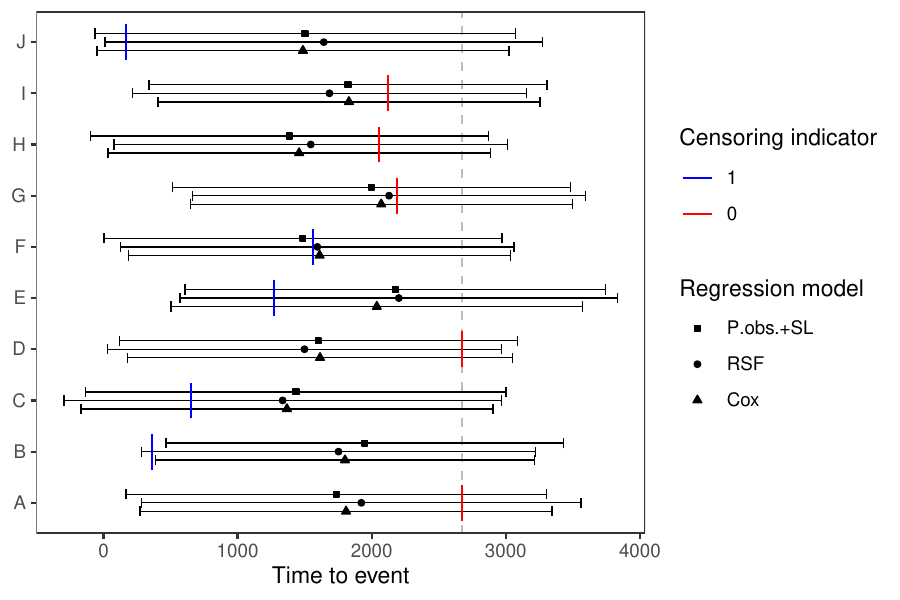}
    \caption{Prediction intervals at the $90\%$ level constructed with the IPCW Rank-One-Out Split Conformal algorithm~\citep[see][]{cwiling_comprehensive_2023} for the Cox model, the RSF and the super learner based on pseudo-observations on the \texttt{colon} dataset. All data are used in the training, and the prediction intervals are displayed for $10$ random machines. The grey dotted line represents the time horizon $\tau = 2672$. The segments, red for censored and blue for uncensored times, are placed at the minimum between the observed times and $\tau$.}
    \label{fig::pred_intvs_colon}
\end{figure}